\PassOptionsToPackage{unicode=true}{hyperref} % options for packages loaded elsewhere
\PassOptionsToPackage{hyphens}{url}
\documentclass[]{article}
\usepackage{lmodern}
\usepackage{amssymb,amsmath}
\usepackage{amsthm}
\usepackage{xpatch}
\newtheorem{theorem}{Theorem}[section]
\newtheorem{lemma}[theorem]{Lemma}

\newtheorem{corollary}[theorem]{Corollary}
\usepackage{ifxetex,ifluatex}
\usepackage{fixltx2e} % provides \textsubscript
\ifnum 0\ifxetex 1\fi\ifluatex 1\fi=0 % if pdftex
  \usepackage[T1]{fontenc}
  \usepackage[utf8]{inputenc}
  \usepackage{textcomp} % provides euro and other symbols
\else % if luatex or xelatex
  \usepackage{unicode-math}
  \defaultfontfeatures{Ligatures=TeX,Scale=MatchLowercase}
\fi
\IfFileExists{upquote.sty}{\usepackage{upquote}}{}
\IfFileExists{microtype.sty}{%
\usepackage[]{microtype}
\UseMicrotypeSet[protrusion]{basicmath} % disable protrusion for tt fonts
}{}
\IfFileExists{parskip.sty}{%
\usepackage{parskip}
}{% else
\setlength{\parindent}{0pt}
\setlength{\parskip}{6pt plus 2pt minus 1pt}
}
\usepackage{hyperref}
\hypersetup{
            pdftitle={Unifying the geometric decompositions of full and trimmed polynomial spaces in finite element exterior calculus},
            pdfauthor={Toby Isaac},
            pdfborder={0 0 0},
            breaklinks=true}
\usepackage{cleveref}
\usepackage[margin=1in]{geometry}
\makeatletter
\def\fps@figure{htbp}
\makeatother

\usepackage{tikz-cd}

\DeclareMathSymbol{\range}{\mathord}{operators}{"3A}

\usepackage[]{biblatex}
\title{Unifying the geometric decompositions of full and trimmed polynomial
spaces in finite element exterior calculus}
\author{Toby Isaac}
\date{}

\begin{document}
\maketitle
\begin{abstract}
Arnold, Falk, \& Winther, in \emph{Finite element exterior calculus,
homological techniques, and applications} (2006), show how to
geometrically decompose the full and trimmed polynomial spaces on
simplicial elements into direct sums of trace-free subspaces and in
\emph{Geometric decompositions and local bases for finite element
differential forms} (2009) the same authors give direct constructions of
extension operators for the same spaces. The two families -- full and
trimmed -- are treated separately, using differently defined
isomorphisms between each and the other's trace-free subspaces and
mutually incompatible extension operators. This work describes a single
operator \(\mathring{\star}_T\) that unifies the two isomorphisms and
also defines a weighted-\(L^2\) norm appropriate for defining
well-conditioned basis functions and dual-basis functionals for
geometric decomposition. This work also describes a single extension
operator \(\dot{E}_{\sigma,T}\) that implements geometric decompositions
of all differential forms as well as for the full and trimmed polynomial
spaces separately.
\end{abstract}

\hypertarget{introduction}{%
\section{Introduction}\label{introduction}}

This introduction, describing the contribution of the subsequent
sections, is intended for those familiar with finite elements but not
with exterior calculus.

\hypertarget{trace-free-scalar-valued-finite-elements}{%
\subsection{Trace-free scalar-valued finite
elements}\label{trace-free-scalar-valued-finite-elements}}

Let \(\mathcal{T}_h\) be a triangulation of a domain
\(\Omega \subset \mathbb{R}^2\), and consider the standard finite
element space \(\mathcal{P}_3(\mathcal{T}_h)\) of continuous,
real-valued functions that are polynomials with degree at most 3 when
restricted to each triangle, \[
\mathcal{P}_3(\mathcal{T}_h) := \{f \in C^0(\Omega): f|_T \in \mathcal{P}_3(T), T \in
\mathcal{T}_h\}.
\]

A standard finite element basis for \(\mathcal{P}_3(\mathcal{T}_h)\)
associates one basis function to each vertex, two to each edge, and one
to each triangle, with each basis function supported only on the
triangles surrounding its associated mesh object.

Let \(T\) be a triangle in \(\mathcal{T}_h\) with vertices at
\(\{v_0, v_1, v_2\}\), and let \(\lambda: T \to \mathbb{R}^3\) be its
barycentric coordinates, the unique affine map satisfying
\(\lambda_i(v_j) = \delta_{ij}\) and \(\sum_i \lambda_i = 1\). The basis
function \(\phi_T\) associated with \(T\), to be fully continuous and
supported only on \(T\), must be zero on the boundary \(T\): we say that
\(\phi_T\) is \emph{trace-free}. Because \(\phi_T\) is in
\(\mathcal{P}_3(T)\) and zero on each of the three edges surrounding
\(T\), \(\phi_T\) must be a multiple of the function \[
\lambda_T := \lambda_0 \lambda_1 \lambda_2.
\] This is the \emph{bubble function} associated with \(T\).

Now let \(E\) be an edge in the mesh with vertices \(\{v_0, v_1\}\), and
let \(\lambda\) be its barycentric coordinates. The two basis functions
\(\{\phi_{E,1},\phi_{E,2}\}\) associated with \(E\), to be continuous
and supported only on the triangles on either side of \(E\), must be
zero at \(v_0\) and \(v_1\). But while the basis function \(\phi_T\) was
determined up to a constant, there is more freedom in what
\(\phi_{E,1}\) and \(\phi_{E,2}\) can be: they must vanish on the
non-shared edges of the surrounding triangles, but when restricted to
\(E\), may satisfy \[
\phi_{E,1}|_E = \lambda_E p_1, \quad \phi_{E,2}|_E = \lambda_E p_2,
\quad\quad[\lambda_E := \lambda_0 \lambda_1]
\] for any basis \(\{p_1,p_2\}\) of \(\mathcal{P}_1(E)\).

The pattern that this example demonstrates is that, for scalar-valued
polynomials on an \(n\)-dimensional simplex \(T\), there is an
isomorphism between \(\mathcal{P}_r(T)\), the polynomials on \(T\) up to
degree \(r\), and the trace-free polynomials up to degree \(r+n+1\), \[
\mathring{\mathcal{P}}_{r+n+1}(T) := \{p\in\mathcal{P}_{r+n+1}(T): p(x)|_{\partial T} = 0\}.
\] This isomorphism can be realized by an operator
\(h_T:\mathcal{P}_r(T) \to \mathring{\mathcal{P}}_{r+n+1}(T)\) defined
by pointwise multiplication with the bubble function, \[
h_T(p) := \lambda_T p.
\] We point out that because this operator is defined pointwise it can
be applied to any smoothly continuous function on \(\overline T\), not
just a polynomial, to construct a trace-free function, and it transforms
any basis of \(\mathcal{P}_r(T)\) into a basis for
\(\mathring{\mathcal{P}}_{r+n+1}(T)\).

\hypertarget{scalar-valued-finite-element-extension-operators}{%
\subsection{Scalar-valued finite element extension
operators}\label{scalar-valued-finite-element-extension-operators}}

To complete the definition of the edge-centered basis function
\(\phi_{E,1}\) from the above example, we must define its values
\(\phi_{E,1}^T\) in each triangle \(T\) adjacent to \(E\): we call this
\emph{extending} \(\phi_{E,1}\) into \(T\). For the basis function to be
a part of \(\mathcal{P}_3(\mathcal{T}_h)\), the extension
\(\phi_{E,1}^T\) must be in \(\mathcal{P}_3(T)\), and it must be
\emph{consistent}: restricting back to \(E\) should match the definition
on the edge, \[ \phi_{E,1}^T|_E =
\phi_{E,1},
\] and restricting it to any \(\tilde{E} \neq E\) should be identically
zero, \[
\phi_{E,1}^T|_{\tilde{E}} = 0,\quad\tilde{E}\neq E.
\]

Suppose \(E\) is the edge bounded by \(\{v_0,v_1\}\) and \(T\) is the
triangle bounded by \(\{v_0,v_1,v_2\}\). Because
\(\phi_{E,1}|_E = \lambda_0 \lambda_1 p_1\) for some
\(p_1\in\mathcal{P}_1(E)\), one way we can define a consistent extension
\(\phi_{E,1}^T\) is by projecting each point in \(T\) to a point in
\(E\) and evaluating \(p_1\), \[
\phi_{E,1}^T (\lambda_0,\lambda_1,\lambda_2) := \lambda_1 \lambda_2 p_1(\lambda_0 + (1/2) \lambda_2, \lambda_1 + (1/2) \lambda_2).
\]

Note that this procedure works not just for extending
\(\mathring{\mathcal{P}}_3 (E)\) into \(\mathcal{P}_3(T)\): it extends
\(\mathring{\mathcal{P}}_r (E)\) into \(\mathcal{P}_r(T)\) for any
\(r\), and its extension of \(\lambda_0\lambda_1 g\) is a consistent
extension for any \(g\) that is smooth on \(\overline{E}\).

\hypertarget{finite-element-exterior-calculus}{%
\subsection{Finite element exterior
calculus}\label{finite-element-exterior-calculus}}

Finite element exterior calculus is a framework that unifies the theory
behind \(H^1\)-conforming finite elements (spanned by fully-continuous
basis functions like \(\mathcal{P}_3(\mathcal{T}_h)\) in the preceding
section) with that of \(H(\mathrm{div})\)-conforming elements (which are
vector-valued and spanned by basis functions whose normal components
only must be continuous on the facets between cells) and
\(H(\mathrm{curl})\)-conforming elements (whose tangential components
only must be continuous on the facets and edges between cells).

In finite element exterior calculus, for each \(k\in\mathbb{N}_0\) there
is a space
\(\Lambda^k(\Omega):=C^\infty(\Omega;\mathrm{Alt}^k \mathbb{R}^n)\) of
\emph{differential \(k\)-forms}, smooth functions that take values in
\(\mathrm{Alt}^k \mathbb{R}^n\), the vector space of \emph{algebraic
\(k\)-forms}, the alternating \(k\)-linear maps that map
\(\bigotimes_{i=1}^k \mathbb{R}^n\) to \(\mathbb{R}\). For each \(k\)
there is a differential operator
\(\mathrm{d}: \Lambda^k(\Omega) \to \Lambda^{k+1}(\Omega)\) that
generalizes \(\mathrm{grad}\) (\(k=0\)), \(\mathrm{curl}\) (\(k=1\)),
and \(\mathrm{div}\) (\(k=n-1\)), and an associated Sobolev space
\(H\Lambda^k(\Omega)\) that is the closure of \(\Lambda^k(\Omega)\)
under the norm \[
\| v \|_{H\Lambda^k}^2 := \int_\Omega |v|^2 + |\mathrm{d}v|^2\ dx.
\]

To construct \(H\Lambda^k(\Omega)\)-conforming finite element spaces,
one must construct piecewise smooth functions that are
\emph{trace-continuous} at the boundaries between cells, where trace
continuity is the generalization of the normal continuity of
\(H(\mathrm{div})\)-conforming spaces and tangential continuity of
\(H(\mathrm{curl})\)-conforming spaces.

Given a simplex \(T\), let \(\mathring\Lambda^k(\overline T)\) be the
trace-free \(k\)-forms that are smoothly continuous on \(\overline T\).
(This is distinct from \(\mathring{\Lambda}^k(T)\) as defined in other
works, which is the subspace of \(\Lambda^k(T)\) of functions compactly
supported in the interior of \(T\).) Given a function space
\(X(T) \subset \Lambda^k(\overline T)\), identifying its trace-free
subspace
\(\mathring{X}(T) := X(T) \cap \mathring{\Lambda}^k(\overline T)\) is
important for finite element construction. These are the basis functions
that can be extended by zero outside of \(T\) while maintaining trace
continuity, so these are the basis functions associated with \(T\) in a
basis for the whole mesh \(\mathcal{T}_h\). Likewise, identifying
\(\mathring{X}(E)\) for a lower-dimensional simplex \(E\) is important
because the basis functions associated with \(E\) are extensions of
\(\mathring{X}(E)\) into the cells surrounding \(E\). Constructing
\(X(T)\) as the direct sum of \(\mathring{X}(T)\) and extensions of
\(\mathring{X}(E)\) for \(E\) in the boundary of \(T\) is a
\emph{geometric decomposition} of \(X(T)\).

There are two primary families of finite elements for the simplex \(T\)
in finite element exterior calculus, chosen for their homological
properties: the full polynomial spaces \(\mathcal{P}_r \Lambda^k(T)\)
and the trimmed spaces \(\mathcal{P}_r^- \Lambda^k(T)\), which are
defined in \cref{sec:notation}. These two families are intertwined,
because each is isomorphic to the family of trace-free subspaces of the
other: \begin{align}
\mathcal{P}_r \Lambda^k (T) &\cong \mathring{\mathcal{P}}_{r + k + 1}^-\Lambda^{n-k}(T);
\label{eq:iso1} \\
\mathcal{P}_r^- \Lambda^k (T) &\cong \mathring{\mathcal{P}}_{r + k}\Lambda^{n-k}(T). \label{eq:iso2}
\end{align}

The work of \textcite{arnold2006finite} is a standard reference to which
we direct the reader for more details on finite element exterior
calculus as well as proofs of isomorphisms \eqref{eq:iso1} and
\eqref{eq:iso2} \cite[theorems 4.16 and
4.22]{arnold2006finite}. In that work, and in subsequent work on the
topic of geometric decompositions
\cite{arnold2009geometric,licht2018basis}, the two isomorphisms are
treated separately, in that two distinct linear operators are defined,
\(h_T^k: \mathcal{P}_r \Lambda^k (T) \stackrel{\sim\,}{\to}\mathring{\mathcal{P}}_{r + k + 1}^-\Lambda^{n-k}(T)\)
and
\(h_T^{k,-}: \mathcal{P}_r^- \Lambda^k (T) \stackrel{\sim\,}{\to}\mathring{\mathcal{P}}_{r + k}\Lambda^{n-k}(T)\)
(though they are anonymous in \autocite{arnold2006finite}), and they are
proved to be bijections. Furthermore, the operators \(h_T^k\) and
\(h_T^{k,-}\) are defined with respect to particular choices of basis
functions, such that is not clear whether \(h_T^{k,-}\) acts pointwise
or not.

Subsequent works by the same authors and others
\cite{arnold2009geometric,licht2018basis} give direct constructions of
the trace-free subspaces that are based on Bernstein polynomials.
Bernstein polynomials have desirable symmetry properties and
low-complexity evaluation algorithms but are known to be ill-conditioned
for finite element operations relative to other bases. Giving these
tradeoffs, it is potentially useful to have a uniform method for
constructing trace-free \(k\)-forms that is not tied to any basis.

In \cite{arnold2009geometric}, the authors define two extension
operators from a boundary simplex \(f\) into a simplex \(g\),
\begin{align}
E_{f,g}^{r,k}: \mathcal{P}_r \Lambda^k (f) &\to \mathcal{P}_r \Lambda^k(g), \label{eq:ext} \\
E_{f,g}^{r,k,-}: \mathcal{P}_r^- \Lambda^k (f) &\to \mathcal{P}_r^- \Lambda^k(g), \label{eq:exttrimed}
\end{align} that have the required consistency for constructing local,
trace-continuous basis functions for the two spaces. These two extension
operators are not only defined differently, but are truly distinct in
that neither is an extension operator for the other: though
\(\mathcal{P}_r^- \Lambda^k (f) \subset \mathcal{P}_r \Lambda^k(f)\),
\(E_{f,g}^{r,k}[\mathcal{P}_r^- \Lambda^k(f)] \not\subset \mathcal{P}_r^- \Lambda^k(g)\),
and though
\(\mathcal{P}_r \Lambda^k (f) \subset \mathcal{P}_{r+1}^- \Lambda^k(f)\),
\(E_{f,g}^{r+1,k,-}[\mathcal{P}_r \Lambda^k(f)] \not\subset \mathcal{P}_r \Lambda^k(g)\).

\hypertarget{unifying-the-geometric-decompositions-of-mathcalp_r-lambdakt-and-mathcalp_r-lambdakt}{%
\subsection{\texorpdfstring{Unifying the geometric decompositions of
\(\mathcal{P}_r^-\Lambda^k(T)\) and
\(\mathcal{P}_r \Lambda^k(T)\)}{Unifying the geometric decompositions of \textbackslash{}mathcal\{P\}\_r\^{}-\textbackslash{}Lambda\^{}k(T) and \textbackslash{}mathcal\{P\}\_r \textbackslash{}Lambda\^{}k(T)}}\label{unifying-the-geometric-decompositions-of-mathcalp_r-lambdakt-and-mathcalp_r-lambdakt}}

\hypertarget{a-unified-construction-of-trace-free-k-forms}{%
\subsubsection{\texorpdfstring{A unified construction of trace-free
\(k\)-forms}{A unified construction of trace-free k-forms}}\label{a-unified-construction-of-trace-free-k-forms}}

This work defines an operator \(\mathring{\star}_T\) from differential
\(k\)-forms to trace-free \((n-k)\)-forms that acts pointwise and
induces the isomorphisms \eqref{eq:iso1} and \eqref{eq:iso2} and the
more general isomorphism for smooth functions on the closure of \(T\):
\[
\begin{aligned}
&\mathring{\star}_T: \Lambda^k(\overline T) \stackrel{\sim\,}{\to}
\mathring{\Lambda}^{n-k}(\overline T);
\\
&\mathring{\star}_T:\mathcal{P}_r \Lambda^k (T) \stackrel{\sim\,}{\to}\mathring{\mathcal{P}}_{r
+ k + 1}^-\Lambda^{n-k}(T);
\\
&\mathring{\star}_T:\mathcal{P}_r^- \Lambda^k (T) \stackrel{\sim\,}{\to}\mathring{\mathcal{P}}_{r
+ k}\Lambda^{n-k}(T).
\end{aligned}
\]

The definition of \(\mathring{\star}_T\) and proof of these claims are
given in the \cref{the-mathringstar_t-operator}. Here we translate
\(\mathring{\star}_T\) into operators which construct normal-free
functions (for \(H(\mathrm{div})\)-conforming finite elements) and
tangent-free functions (for \(H(\mathrm{curl})\)-conforming finite
elements).

\paragraph{Normal-free.}

Let \(\{E_i\}_{i=1}^{n_e}\) be the edges around \(T\) (\(n_e=3\) if
\(n=2\), \(n_e=6\) if \(n=3\)), let \(e_i = v_{i,1} - v_{i,0}\), where
\(\{v_{i,0},v_{i,1}\}\) are the vertices of edge \(E_i\), and let
\(|T|\) be the \(n\)-dimensional volume of \(T\). Then for any
\(u\in C^\infty(\overline T; \mathbb{R}^n)\), \[
\mathring{\star}_T u = \frac{1}{n!|T|}\sum_{i=1}^{n_e} \langle u,  e_i \rangle
\lambda_{E_i} e_i
\] is a function with normal-free boundary trace. In this translation of
\(\mathring{\star}_T\), the input vector-valued function is treated as a
function in \(\Lambda^1(\overline T)\), while the output is treated as a
function in \(\mathring{\Lambda}^{n-1}(\overline T)\). This operator
transforms any basis for a Nédélec edge element of the first (second)
kind into a basis for the normal-free subspace of a Nédélec face element
of the second (first) kind.

\paragraph{Tangent-free.}

Let \(\{F_i\}_{i=0}^n\) be the facets around \(T\), let \(\nu_i\) be the
unit normal vector for facet \(i\), and let \(\tilde{\nu}_i\) be
\(\nu_i\) scaled by \((n-1)!|F_i|\), the determinant of a matrix formed
from any set of \(n-1\) edge vectors around \(F_i\) (in 2D,
\(\tilde \nu_i\) is a 90 degree rotation of \(v_{i,1} - v_{i,0}\); in
3D, \(\tilde \nu_i\) is the cross product
\((v_{i,1} - v_{i,0})\times(v_{i,2} - v_{i,0})\)). Then for any
\(u \in C^\infty(\overline T; \mathbb{R}^n)\), \[
\mathring{\star}_T u = \frac{1}{n!|T|}\sum_{i=0}^{n} \langle u, \tilde\nu_i \rangle
\lambda_{F_i} \tilde\nu_i 
\] is a function with tangent-free boundary trace. This translation of
\(\mathring{\star}_T\) is the opposite of the last: the input is treated
as a function in \(\Lambda^{n-1}(\overline T)\) and the output in
\(\mathring{\Lambda}^1(\overline T)\). This operator transforms any
basis for a Nédélec face element of the first (second) kind into a basis
for the tangent-free subspace of a Nédélec edge element of the second
(first) kind.

\hypertarget{a-unified-extension-of-trace-free-k-forms}{%
\subsubsection{\texorpdfstring{A unified extension of trace-free
\(k\)-forms}{A unified extension of trace-free k-forms}}\label{a-unified-extension-of-trace-free-k-forms}}

This work also defines an extension operator
\(\dot{E}_{f,g}:\mathring{\Lambda}^k(\overline{f}) \to \Lambda^k(\overline{g})\)
that acts pointwise and is a consistent extension operator for full and
trimmed polynomials and indeed all differential \(k\)-forms: for every
simplex \(g\) in the triangulation \(\mathcal{T}_h\), \[
\begin{aligned}
\Lambda^k(\overline{g}) &= \bigoplus_{f\in\Delta(g)} \dot{E}_{f,g} [ \mathring{\Lambda}^k(\overline f)], \\
\mathcal{P}_r \Lambda^k(\overline{g}) &= \bigoplus_{f\in\Delta(g)} \dot{E}_{f,g} [ \mathring{\mathcal{P}}_r\Lambda^k(\overline f)], \\
\mathcal{P}_r^- \Lambda^k(\overline{g}) &= \bigoplus_{f\in\Delta(g)} \dot{E}_{f,g} [ \mathring{\mathcal{P}}_r^-\Lambda^k(\overline f)].
\end{aligned}
\] We do not translate \(\dot{E}_{f,g}\) into \(H(\mathrm{curl})\) or
\(H(\mathrm{div})\) notation here, but do note that it is based on a
decomposition that generalizes the decomposition present in the
hierarchical bases for \(\mathcal{P}_r \Lambda^k(T)\) of
\textcite{ainsworth2003hierarchic}. This work shows that, although
trimmed polynomials are not explicitly represented in those bases, their
extensions remain trimmed polynomials of the same order.

\hypertarget{preliminaries-and-notation}{%
\section{Preliminaries and notation}\label{preliminaries-and-notation}}

\label{sec:notation}

\paragraph{Integer maps, permutations, and multi-indices.}

We let \([a\range{b}]\) be the set \(\{a,a+1,\dots,b\}\) if \(b \geq a\)
and \([a\range{b}] = \emptyset\) otherwise. For a map
\(\rho: [a\range{b}] \to S\) and \(i\in [a\range{b}]\) we define
\(\rho \setminus \rho(i): [a\range{b-1}] \to S\) by \[
((\rho \setminus \rho(i))(a), \dots, (\rho \setminus \rho(i))(b-1)) = (\rho(a), \dots, \hat\rho(i), \dots \rho(b)),
\] where \(\hat\rho(i)\) is the omitted range index. Given any map
\(\rho: M \to S\) we let \([\![\rho ]\!]\) be the range set
\(\{\rho(i)\}_{i\in M}\).

Given any two subsets \(M\) and \(S\) of \(\mathbb{N}_0\) we let
\(\Sigma(M,S)\) be the set of increasing maps from \(M\) to \(S\). We
adopt the convention that \(\Sigma(\emptyset,S)\) contains only the
empty map \(\emptyset \to \emptyset\). If \(M\) is a subset of \(S\) and
\(\rho \in \Sigma(M,S)\), we define its complement to be
\(\rho^* \in \Sigma(S \setminus M, S)\) such that
\([\![\rho ]\!] \cup [\![\rho^* ]\!] = S\). The two complementary maps
together define a permutation: we let
\(\mathop{\mathrm{sign}}(\rho) \in \{-1,1\}\) be \(1\) if that
permutation is even and \(-1\) if it is odd.

When treating a map into \(\alpha: S \to \mathbb{N}_0\) as a
multi-index, we use the standard notation
\(|\alpha| = \sum_{i\in S} \alpha(i).\)

\paragraph{Exterior algebra.}

We let \(\mathrm{Alt}^k V\) denote the alternating \(k\)-linear forms,
or \emph{algebraic \(k\)-forms}, on the vector space \(V\). By
``alternating'' we mean that \(\omega \in \mathrm{Alt}^k V\) satisfies
\[
\omega(v_1,\dots,v_i,\dots,v_j,\dots,v_k)
=
-\omega(v_1,\dots,v_j,\dots,v_i,\dots,v_k),\quad i,j\in [1\range{k}].
\]

Given \(\omega \in \mathrm{Alt}^j V\) and \(\mu \in\mathrm{Alt}^k V\),
the \emph{wedge product} \(\omega \wedge \mu \in \mathrm{Alt}^{j+k} V\)
is defined by \[
(\omega \wedge \mu)(v_1,\dots,v_{j+k}) := \sum_{\sigma \in \Sigma([1\range{j}],[1\range{j+k}])} \mathop{\mathrm{sign}}(\sigma) \omega(v_{\sigma(1)}, \dots, v_{\sigma(j)}) \mu(v_{\sigma^*(j+1)}, \dots, v_{\sigma^*(j+k)}).
\] The wedge product is bilinear and anti-commuting,
\(\mu \wedge \omega = (-1)^{jk} \omega \wedge \mu\).

Given a linear map \(J : V \to W\), the \emph{pullback} is
\(J^*:\mathrm{Alt}^k W \to \mathrm{Alt}^k V\) defined by \[
(J^* \omega)(v_1, \dots, v_k) := \omega (J v_1, \dots, J v_k), \quad \omega \in \mathrm{Alt}^k W,
\] which distributes over the wedge product
\(J^*(\omega \wedge \mu) = (J^* \omega) \wedge (J^* \mu)\).

Given \(\omega \in \mathrm{Alt}^k V\) and \(v\in V\), the \emph{interior
product} \(\omega\lrcorner v \in \mathrm{Alt}^{k-1}V\) is defined by \[
(\omega\lrcorner v)(v_1,\dots, v_{k-1}) := \omega(v,v_1,\dots,v_{k-1}).
\] The interior product follows the product rule \[
(\omega \wedge \mu)\lrcorner v
= (\omega \lrcorner v) \wedge \mu + (-1)^j \omega \wedge (\mu \lrcorner v),
\quad \omega \in \mathrm{Alt}^j V, \mu \in \mathrm{Alt}^k V.
\]

The \emph{inner product} on \(\mathrm{Alt}^k V\) when \(V\) has its own
inner product is defined using any orthonormal basis \(\{q_k\}_{k=1}^n\)
of \(V\) as \[
\langle \omega, \mu \rangle := \sum_{\rho \in \Sigma([1\range{k}],[1\range{n}]))}
\omega(q_{\rho(1)}, \dots, q_{\rho(k)})
\mu(q_{\rho(1)}, \dots, q_{\rho(k)}), \quad \omega,\mu \in \mathrm{Alt}^k V.
\] If the space not only has an inner product but a sign convention that
the ordering \((q_1, \dots, q_k)\) is positively oriented, then the
\emph{volume form} \(\mathrm{vol}\in\mathrm{Alt}^n V\) is defined by \[
\mathrm{vol}(q_1, \dots, q_n) := 1.
\]

The \emph{Hodge star operator}
\(\star: \mathrm{Alt}^k V \stackrel{\sim\,}{\to}\mathrm{Alt}^{n-k} V\)
is an isometry defined by \[
\omega \wedge \mu := \langle \star \omega, \mu \rangle\mathrm{vol}, \quad \omega\in\mathrm{Alt}^k V, \mu
\in \mathrm{Alt}^{n-k}V.
\]

\paragraph{Exterior calculus.}

Exterior calculus can be defined for general manifolds: in this work we
only need to consider subsets \(\Omega\) of \(\mathbb{R}^n\) which have
the same tangent space \(T_x\Omega\) for each \(x\in \Omega\), either
because they are \(n\)-dimensional and \(T_x \Omega = \mathbb{R}^n\) or
because they are \(d\)-dimensional subsets of \(d\)-dimensional
hyperplanes.

The set of smooth, \(\mathrm{Alt}^k T_x \Omega\)-valued function on
\(\Omega \subset \mathbb{R}^n\) are \emph{differential \(k\)-forms} and
are denoted \(\Lambda^k(\Omega)\). The \emph{exterior derivative} on
\(k\)-forms,
\(\mathrm{d}: \Lambda^k(\Omega) \to \Lambda^{k+1}(\Omega)\), is defined
by \[
(\mathrm{d}\omega)_x(v_1, \dots, v_{k+1}) := \sum_{j=1}^{k+1} (-1)^{j+1}
\nabla_x (\omega_x(v_1, \dots, \hat{v}_j, \dots, v_{k+1})) \cdot v_j,
\] where again \(\hat{v}_j\) is the omitted argument. The exterior
derivative applied to each of the coordinate functions
\(x_1, \dots, x_n\) is a constant \(1\)-form, so we consider
\(\mathrm{d}x_i\) an algebraic 1-form for each \(i\). The set
\(\{\mathrm{d}x_i\}_{i=1}^n\) forms a basis of algebraic \(k\)-forms
that are dual to the canonical basis \(\{e_i\}_{i=1}^n\) of
\(\mathbb{R}^n\), \[
\mathrm{d}x_i(e_j) = \delta_{ij}, \quad i,j \in [1\range{n}].
\] Given a sequence \(\rho: [a\range{b}] \to S\), not necessarily
increasing, we let \(x_{\rho} := \prod_{i=a}^b x_{\rho(i)}\) and
\[(\mathrm{d}x)_{\rho} := \mathrm{d}x_{\rho(a)} \wedge \dots \wedge \mathrm{d}x_{\rho(b)},\]
with \(x_\rho := (\mathrm{d}x)_\rho := 1\) if
\([a\range{b}] = \emptyset\). With this notation the canonical
orthonormal basis of \(\mathrm{Alt}^k \mathbb{R}^n\) in terms of the
coordinate functions is written as
\(\{(\mathrm{d}x)_{\rho}\}_{\rho\in\Sigma([1\range{k}], [1\range{n}])}\),
which we call the set of \emph{coordinate \(k\)-forms}.

If \(\phi:T \to U\) is a smooth map, then
\(\phi^*: \Lambda^k(U) \to \Lambda^k(T)\) is the pullback for
differential \(k\)-forms, defined by \[
(\phi^* \omega)_x := 
(D\phi)_x^* \omega_{\phi(x)}, \quad \omega \in \Lambda^k(U),
\] where \((D\phi)_x\) is the Jacobian of \(\phi\) at \(x\).

We let \(\kappa_x : \Lambda^k(\Omega) \to \Lambda^{k-1}(\Omega)\) be the
\emph{Koszul operator} centered at \(x\), which is defined by \[
(\kappa_x \omega)_y := \omega_y \lrcorner (y-x).
\] An important fact for this work is that the Koszul operator commutes
with the pullback of an affine map up to a constant interior product:
given affine \(\phi: T \to \hat{T}\),
\begin{equation}\label{eq:kozsul-commute}
\kappa_x (\phi^* \omega)
=
\phi^*(\kappa_{\hat{x}} \omega - \omega \lrcorner (\phi(x) - \hat x)).
\end{equation}

\paragraph{Oriented simplices.}

In all that follows \(T\) is a positively oriented \(n\)-simplex in
\(\mathbb{R}^n\), \(|T|\) is its volume and \(\{\lambda_i\}_{i=0}^n\)
are its barycentric coordinates. We define \(\lambda_\rho\) and
\((\mathrm{d}\lambda)_{\rho}\) similarly to \(x_\rho\) and
\((\mathrm{d}x)_{\rho}\), and we refer to \((\mathrm{d}\lambda)_{\rho}\)
as a \emph{barycentric \(k\)-form}. Given maps
\(\rho: S \to [0\range{n}]\) and \(\alpha: S \to \mathbb{N}_0\) with the
same domain \(S\), a barycentric monomial is a polynomial in
\(\mathcal{P}_{|\alpha|}(T)\) defined by \[
(\lambda_\rho)^\alpha := \Pi_{i \in S} \lambda_{\rho(i)}^{\alpha(i)}.
\]

A positively oriented simplex has the property
\cite[Section 4.1]{arnold2006finite} that for any permutation
\(\pi:[0\range{n}] \to [0\range{n}]\) and any \(i\in [0\range{n}]\) \[
%\begin{equation}\label{eq:oriented}
(\mathrm{d}\lambda)_{\pi \setminus \pi(i)}
 = (-1)^i\frac{\mathop{\mathrm{sign}}\pi}{n!|T|}\mathrm{vol}.
%\end{equation}
\]

\paragraph{Boundary simplices.}

For each \(\rho \in \Sigma([0\range{d}],[0\range{n}])\), the
\(d\)-simplex with vertices \(\{v_{\rho(0)},\dots,v_{\rho(d)}\}\) in the
boundary of \(T\) is \(f_\rho\). If
\(i_\rho: f_{\rho} \hookrightarrow T\) is the inclusion map, then the
trace operator
\(\mathrm{Tr}_{\rho}:\Lambda^k(\overline T) \to \Lambda^k(\overline {f_\rho})\)
is its pullback, \(\mathrm{Tr}_{\rho} := i_\rho^*\). In
\cref{sec:extension}, when we need the trace of
\(\omega\in\Lambda^k(\overline{f_\rho})\) on another boundary simplex
\(f_\sigma\), we denote it \(\mathrm{Tr}_{\rho,\sigma}\omega\).

The trace operator lets us more formally define the space
\(\mathring{\Lambda}^k(\overline T)\) from the introduction as \[
\mathring{\Lambda}^k(\overline T):= \{\omega\in \Lambda^k(\overline T):
\mathrm{Tr}_{\rho} u = 0, f_\rho \neq T\}.
\]

\paragraph{Whitney forms and trimmed polynomials.}

The Whitney form \(\phi_{\rho}\in \mathcal{P}_1 \Lambda^d (T)\)
associated with \(f_\rho\) is \[
\phi_{\rho} := \sum_{i=0}^d (-1)^i \lambda_{\rho(i)} (\mathrm{d}\lambda)_{\rho \setminus \rho(i)},
\] and the trimmed polynomials are spanned by the product of Whitney
forms and scalar polynomials, \[
\mathcal{P}_r^- \Lambda^k(T) = \mathrm{span}\{ a_\rho \phi_\rho: a_\rho \in
\mathcal{P}_{r-1}(T), \rho \in \Sigma([0\range{k}],[0\range{n}])\}.
\] An important fact about the trimmed polynomials in this work is that
\(\mathcal{P}_r^- \Lambda^k(T)\) is the largest subspace of
\(\mathcal{P}_r \Lambda^k(T)\) such that
\(\kappa_x p \in \mathcal{P}_r \Lambda^{k-1}(T)\) for every
\(p \in \mathcal{P}_r^- \Lambda^k(T)\) and every \(x\).

\hypertarget{trace-free-operators}{%
\section{Trace-free operators}\label{trace-free-operators}}

\hypertarget{previous-trace-free-operators-and-bases}{%
\subsection{Previous trace-free operators and
bases}\label{previous-trace-free-operators-and-bases}}

\label{sec:prev-trace-free}

We can now define the trace-free operators \(h_T^k\) and \(h_T^{k,-}\)
given in \autocite{arnold2006finite} that induce isomorphisms
\eqref{eq:iso1} and \eqref{eq:iso2}: \begin{align}
&h_T^k: \sum_{\substack{\rho\in\Sigma([0\range{n-k}],[0\range{n}]) \\ \rho(0) = 0}} a_\rho (\mathrm{d}
\lambda)_{\rho^*}
\mapsto
\sum_{\substack{\rho\in\Sigma([0\range{n-k}],[0\range{n}]) \\ \rho(0) = 0}} 
a_\rho \lambda_{\rho^*} \phi_{\rho} &[a_\rho \in \mathcal{P}_r(T)]; \label{eq:map1}\\
&h_T^{k,-}: \sum_{\rho\in\Sigma([0\range{k}],[0\range{n}])} a_\rho \phi_\rho
\mapsto
\sum_{\rho\in\Sigma([0\range{k}],[0\range{n}])} a_\rho \lambda_{\rho} (\mathrm{d}\lambda)_{\rho*}
&[a_\rho = a_\rho(\lambda_{\rho(0)},\lambda_{\rho(0)+1},\dots,\lambda_n) \in
\mathcal{P}_{r-1}(T)]. \label{eq:map2}
\end{align} The set
\(\{(\mathrm{d}\lambda)_{\rho^*}\}_{\rho\in\Sigma([0\range{n-k}],[0\range{n}]), \rho(0) = 0}\)
is a basis for \(\mathrm{Alt}^k \mathbb{R}^n\), so \(h_T^k\) can be
computed pointwise, but the set
\(\{(\phi_\rho)_x\}_{\rho \in \Sigma([0\range{k}],[0\range{n}])}\) has
\(\binom{n+1}{k}\) non-zero elements for \(x\) in the interior of \(T\)
while \(|\mathrm{Alt}^k \mathbb{R}^n| = \binom{n}{k}\), so the values of
the coefficients
\(\{a_\rho\}_{\rho\in \Sigma([0\range{k}],[0\range{n}])}\) at a point
are overdetermined in \cref{eq:map2}, making it appear that
\(h_T^{k,-}\) is not a pointwise operator. We show that this appearance
is deceptive because both \(h_T^k\) and \(h_T^{k,-}\) are equivalent to
the pointwise operator \(\mathring{\star}_T\) we define in
\cref{the-mathringstar_t-operator}.

These maps directly lead to the constructive bases given for the
trace-free subspace in \cite{arnold2009geometric,licht2018basis} in the
following way. In \eqref{eq:map1}, we can express
\(a_\rho \in \mathcal{P}_r(T)\) as a polynomial of \(n+1\) arguments in
the barycentric coordinates,
\(a_\rho = a_\rho(\lambda_0, \dots, \lambda_n)\), so that both
\eqref{eq:map1} and \eqref{eq:map2} use polynomials of barycentric
coordinates. If one expands each \(a_\rho\) in terms of a monomial basis
of the barycentric coordinates (which is equivalent up to a
multiplicative constant to the Bernstein polynomial basis), the right
hand side of \eqref{eq:map1} defines a basis for
\(\mathring{\mathcal{P}}_{r+k+1}^-\Lambda^{n-k}(T)\) where each basis
function is the product of a barycentric monomial and a Whitney form,
and the right hand side of \eqref{eq:map2} defines a basis for
\(\mathring{\mathcal{P}}_{r+k}\Lambda^{n-k}(T)\) which each basis
function is the product of a barycentric monomial and an \((n-k)\)-form
\((\mathrm{d}\lambda)_{\rho^*}\) for some
\(\rho \in \Sigma([0\range{k}],[0\range{n}])\).

\hypertarget{the-star_t-operator}{%
\subsection{\texorpdfstring{The \(\star_T\)
operator}{The \textbackslash{}star\_T operator}}\label{the-star_t-operator}}

The duality between \(\mathrm{Alt}^k \mathbb{R}^n\) and
\(\mathrm{Alt}^{n-k} \mathbb{R}^n\) is clearly relevant to the
isomorphisms \eqref{eq:iso1} and \eqref{eq:iso2}, but the Hodge star
operator
\(\star: \mathrm{Alt}^k \mathbb{R}^n \stackrel{\sim\,}{\to}\mathrm{Alt}^{n-k} \mathbb{R}^n\)
does not appear to play any role. While this is surprising at first, it
becomes clear that the Hodge star cannot be used because it is not
affine invariant: given a bijection
\(\phi: T \stackrel{\sim\,}{\to}\hat{T}\) between simplices, in general
\(\star \neq \phi^* \circ \star \circ \phi^{-*}\).

For each simplex \(T\), we define a bijection between
\(\mathrm{Alt}^k \mathbb{R}^n\) and \(\mathrm{Alt}^{n-k} \mathbb{R}^n\)
based instead on the barycentric coordinates of \(T\):
\begin{equation}\label{eq:start}
\star_T: \omega \mapsto \frac{n!|T|}{\sqrt{n+1}} \sum_{\rho \in
\Sigma([0\range{n-k-1}],[0\range{n}])} \star(\omega \wedge (\mathrm{d}\lambda)_{\rho}) (\mathrm{d}
\lambda)_{\rho}.
\end{equation}

\begin{lemma}\label{lem:star_t-bijection}
  $\star_T$ is a bijection, $\star_T : \mathrm{Alt}^k \mathbb{R}^n \stackrel{\sim\,}{\to}
  \mathrm{Alt}^{n-k} \mathbb{R}^n$.
\end{lemma}

\begin{proof}
  We can define a symmetric bilinear form on $\mathrm{Alt}^k \mathbb{R}^n$,
  $$
    \langle \omega, \mu \rangle_{\star_T} := \star(\omega \wedge \star_T \mu)
    = 
    \frac{n!|T|}{\sqrt{n+1}} \sum_{\rho \in
    \Sigma([0\range{n-k-1}],[0\range{n}])} \star(\omega \wedge (\mathrm{d}\lambda)_{\rho})\ {\star(\mu
    \wedge (\mathrm{d}\lambda)_{\rho})}.
  $$
  If $\star_T \omega = 0$, then $\langle \omega, \omega \rangle_{\star_T} = 0$,
  and so $\star(\omega \wedge (\mathrm{d}\lambda)_{\rho}) = 0$ for all $\rho \in
  \Sigma([0\range{n-k-1}],[0\range{n}])$.  These $(n-k)$-forms span $\mathrm{Alt}^{n-k} \mathbb{R}^n$, and
  $\mathrm{Alt}^{n-k} \mathbb{R}^n \cong (\mathrm{Alt}^n \mathbb{R}^n)^*$, so $\omega = 0$.
\end{proof}

Because \(\star_T\) is defined by a simplex instead of being universal,
it can have a useful form of affine invariance.

\begin{lemma}\label{lem:affinv}
  Given an affine bijection $\phi:T \stackrel{\sim\,}{\to}\hat{T}$, $\star_T = (D\phi)^* \circ \star_{\hat{T}} \circ (D\phi)^{-*}$.
\end{lemma}

\begin{proof}
  If $\hat{\lambda}$ are the barycentric coordinates of $\hat{T}$, then $(\mathrm{d}
  \hat{\lambda})_{\rho} = (D\phi)^{-*} (\mathrm{d}\lambda)_{\rho}$ and $(D\phi)^{-*} \mathrm{vol}= (|T|
  / |\hat{T}|) \mathrm{vol}$.  Hence
  $$
  \begin{aligned}
  &((D\phi)^* \circ \star_{\hat{T}}\circ (D\phi)^{-*}) (\omega) \\
  &=
  (D\phi)^* \left(
    \frac{n! |\hat{T}|}{\sqrt{n+1}}
    \sum_{\rho\in\Sigma([0\range{n-k-1}],[0\range{n}])}
      \star((D\phi)^{-*}\omega \wedge (\mathrm{d}\hat{\lambda})_{\rho})\ 
      (\mathrm{d}\hat{\lambda})_{\rho}
  \right)
  \\
  &=
  (D\phi)^* \left(
    \frac{n! |\hat{T}|}{\sqrt{n+1}}
    \sum_{\rho\in\Sigma([0\range{n-k-1}],[0\range{n}])}
      \star((D\phi)^{-*}\omega \wedge (D\phi)^{-*}(\mathrm{d}\lambda)_{\rho})\ 
      (D\phi)^{-*}(\mathrm{d}\lambda)_{\rho}
  \right)
  \\
  &=
    \frac{n! |\hat{T}|}{\sqrt{n+1}}
    \sum_{\rho\in\Sigma([0\range{n-k-1}],[0\range{n}])}
      \star(D\phi)^{-*}(\omega \wedge (\mathrm{d}\lambda)_{\rho})\ 
      (\mathrm{d}\lambda)_{\rho}
  \\
  &=
    \frac{n! |T|}{\sqrt{n+1}}
    \sum_{\rho\in\Sigma([0\range{n-k-1}],[0\range{n}])}
      \star(\omega \wedge (\mathrm{d}\lambda)_{\rho})\ (\mathrm{d}\lambda)_{\rho}
  \\
  &= \mathring{\star}_T \omega.
  \end{aligned}
  $$
\end{proof}

\begin{lemma}
  If $T_{\mathrm{eq}}$ is an equilateral simplex with edge length $\sqrt{2}$,
  then $\star_{T_{\mathrm{eq}}} = \star$.
\end{lemma}

\begin{proof}
If we take the convention that the canonical basis of $\mathbb{R}^{n+1}$ is numbered
$e_0, \dots, e_n$, then for every simplex the barycentric $k$-forms are
pullbacks of the coordinate $k$-forms,
$$(\mathrm{d}\lambda)_{\rho} = (D\lambda)^* (\mathrm{d}x)_\rho, \quad \rho\in\Sigma([0\range{k}],[0\range{n}]).$$ In the case of $T_{\mathrm{eq}}$, the Jacobian $D\lambda$ is
isometric: $\|(D\lambda)v\|_2 = \|v_2\|$ for all $v \in \mathbb{R}^n$. This is clear
when we think of $\lambda$ as mapping $T$ to the standard barycentric simplex,
which is the simplex in the positive orthant of $\mathbb{R}^{n+1}$ connecting the $e_i$ unit vectors, which is also equilateral and has the same edge length.

Because $(D\lambda)$ is isometric, the $\binom{n+1}{k+1} \times \binom{n+1}{k+1}$ matrix of inner products between barycentric $k$-forms,
$$
M_{\rho,\sigma} := \langle (\mathrm{d}\lambda)_{\rho}, (\mathrm{d}\lambda)_{\sigma} \rangle = \langle(D\lambda)^*(\mathrm{d}x)_\rho,(D\lambda)^*(\mathrm{d}x)_\sigma\rangle,\quad\rho,\sigma\in\Sigma([0\range{k}], [0\range{n}]),
$$
is an orthogonal projection matrix for every $k$.
This implies that the set $\{(\mathrm{d}\lambda)_{\rho}\}_{\rho\in\Sigma([0\range{k}],[0\range{n}])}$
is a normalized tight frame \cite[Theorem 2.5]{waldron2018introduction}, that
is
$$
\omega = \sum_{\rho\in\Sigma([0\range{k}],[0\range{n}])} \langle \omega, (\mathrm{d}\lambda)_{\rho}
\rangle (\mathrm{d}\lambda)_{\rho}, \quad \omega \in \mathrm{Alt}^k \mathbb{R}^n.
$$
To compute $\star_{T_{\mathrm{eq}}}$ for $\omega\in\mathrm{Alt}^k\mathbb{R}^n$, we apply the above fact to the $(n-k)$-forms, and note that $|T_{\mathrm{eq}}| = \sqrt{n+1}/{n!}$ so that the leading constant in \eqref{eq:start} is cancelled, to see
$$
\star_{T_{\mathrm{eq}}} \omega = \sum_{\rho\in\Sigma([0\range{n-k-1}],[0\range{n}])}\star(\omega
\wedge (\mathrm{d}\lambda)_{\rho}) (\mathrm{d}\lambda)_{\rho}
= \sum_{\rho\in\Sigma([0\range{n-k-1}],[0\range{n}])}\langle \star\omega, (\mathrm{d}\lambda)_{\rho}
\rangle (\mathrm{d}\lambda)_{\rho}
= \star \omega.
$$
\end{proof}

These lemmas imply that, like \(\star\) itself,
\(\star_T \circ \star_T = (-1)^{k(n-k)}\).

\begin{corollary}\label{cor:twice}
$\star_T \circ \star_T = (-1)^{k(n-k)}.$
\end{corollary}

\begin{proof}
Let $\phi:T \stackrel{\sim\,}{\to}T_{\mathrm{eq}}$ be an affine bijection.  Then
$$
\star_T \circ \star_T = (D\phi)^* \circ \star_{T_\mathrm{eq}} \circ (D\phi)^{-*} \circ
(D\phi)^* \circ \star_{T_\mathrm{eq}} \circ (D\phi)^{-*} = (D\phi)^* \circ \star \circ \star
\circ (D\phi)^{-*} = (-1)^{k(n-k)}.
$$
\end{proof}

\hypertarget{the-mathringstar_t-operator}{%
\subsection{\texorpdfstring{The \(\mathring{\star}_T\)
operator}{The \textbackslash{}mathring\{\textbackslash{}star\}\_T operator}}\label{the-mathringstar_t-operator}}

Having defined an isomorphism \(\star_T\) between algebraic \(k\)- and
\((n-k)\)-form with the desired variance, we are now prepared to define
an similar isomorphism between differential \(k\)- and trace-free
\((n-k)\)-forms. A similar construction to \eqref{eq:start} defines a
linear operator
\(\mathring{\star}_T: \Lambda^k(\overline T) \to \mathring{\Lambda}^{n-k}(\overline T)\)
by multiplying the forms summed in \eqref{eq:start} by their
complementary bubble functions. We define
\begin{equation}\label{eq:ringo}
\mathring{\star}_T: \omega \mapsto
  n!|T|
    \sum_{\rho \in \Sigma([0\range{n-k-1}],[0\range{n}])}
      \star(\omega \wedge (\mathrm{d}\lambda)_{\rho})
      \lambda_{\rho^*}
      (\mathrm{d}\lambda)_{\rho}.
\end{equation}

\begin{lemma}
  $\mathring{\star}_T$ is affine invariant: given affine bijection $\phi:T \stackrel{\sim\,}{\to}\hat T$,
  $\mathring{\star}_T = \phi^* \circ \mathring{\star}_{\hat T} \circ \phi^{-*}.$
\end{lemma}

\begin{proof}
  The proof is essentially the same as for lemma~\ref{lem:affinv}.
\end{proof}

\begin{lemma}\label{lem:injection}
  $\mathring{\star}_T: \Lambda^k(\overline T) \to
  \mathring{\Lambda}^{n-k} (\overline T)$ is an injection.
\end{lemma}

\begin{proof}
  To show that the range of $\mathring{\star}_T$ is in 
  $\mathring{\Lambda}^{n-k} (\overline T)$ it is sufficient to show
  that $\lambda_{\rho^*} (\mathrm{d}\lambda)_{\rho} \in \mathring{\Lambda}^{n-k}
  (\overline T)$ for each $\rho\in \Sigma([0\range{n-k-1}],[0\range{n}])$.  This is so because
  $\mathrm{Tr}_\sigma \lambda_{\rho^*} = 0$ if $[\![\rho^* ]\!] \not \subseteq
  [\![\sigma ]\!]$ and $\mathrm{Tr}_\sigma (\mathrm{d}\lambda)_{\rho} = 0$ if $[\![\rho ]\!]
  \not \subseteq [\![\sigma ]\!]$.  Therefore $\mathrm{Tr}_\sigma \lambda_{\rho^*} (\mathrm{d}
  \lambda)_{\rho} = 0$ for any $[\![\sigma ]\!] \not \supseteq [\![\rho ]\!] \cup
  [\![\rho^* ]\!]$, that is for any $f_\sigma \neq T$.

  To show injectivity, as in lemma~\ref{lem:star_t-bijection}, we can define a
  symmetric bilinear form
  \begin{equation}\label{eq:ringot-inner-product}
  \langle \omega, \mu \rangle_{\mathring{\star}_T} := \int_T \omega \wedge \mathring{\star}_T \mu
  =
  n!|T| \int_T \mathrm{vol}
    \sum_{\rho \in \Sigma([0\range{n-k-1}],[0\range{n}])}
      \lambda_{\rho^*}\ 
      {\star(\omega \wedge (\mathrm{d}\lambda)_{\rho})}\ 
      {\star(\mu \wedge (\mathrm{d}\lambda)_\rho)}.
  \end{equation}
  If $\mathring{\star}_T \omega = 0$, then $\langle \omega, \omega \rangle_{\mathring{\star}_T} = 0$, and
  for every $x$ in $T$
  $$
    \sum_{\rho \in \Sigma([0\range{n-k-1}],[0\range{n}])} \lambda_{\rho^*}(\star(\omega_x \wedge (\mathrm{d}\lambda)_{\rho}))^2 = 0.
  $$
  For every $\rho \in \Sigma([0\range{n-k-1}],[0\range{n}])$ and every $x$ in the interior of $T$ where $\lambda_{\rho^*} > 0$ this implies that
  $\star(\omega_x \wedge (\mathrm{d}\lambda)_{\rho}) = 0$ and so $\omega_x = 0$.  By continuity,
  we conclude $\omega = 0$.
\end{proof}

The operator \(\mathring{\star}_T\) has an important property, analogous
to corollary \ref{cor:twice}, that applying \(\mathring{\star}_T\) twice
is like multiplying by \((-1)^{k(n-k)} \lambda_T\).

\begin{lemma}\label{lem:ringotwice}
$\mathring{\star}_T \circ \mathring{\star}_T = (-1)^{k(n-k)} \lambda_T$.
\end{lemma}

\begin{proof}
Let $\omega \in \Lambda^k(\overline T)$ be given and first expand $\mathring{\star}_T \mathring{\star}_T \omega:$
$$
\mathring{\star}_T \mathring{\star}_T \omega
= 
(n!|T|)^2 \sum_{\substack{\rho \in \Sigma([0\range{n-k-1}],[0\range{n}]) \\ \sigma \in \Sigma([0\range{k-1}],[0\range{n}])}}
\lambda_{\rho^*}\lambda_{\sigma^*}\ 
{\star(\omega \wedge (\mathrm{d}\lambda)_\rho)}\ {\star((\mathrm{d}\lambda)_{\rho} \wedge (\mathrm{d}\lambda)_{\sigma})}\ (\mathrm{d}\lambda)_{\sigma}.
$$
If $[\![\rho ]\!] \cap [\![\sigma ]\!] \neq \emptyset$, then $(\mathrm{d}
\lambda)_{\rho} \wedge (\mathrm{d}\lambda)_{\sigma} = 0$. So for each nonzero summand
there is a map $\tau \in \Sigma([0\range{n-1}],[0\range{n}])$ and a map $\hat{\rho}\in
\Sigma([0\range{{n-k-1}}],[0\range{{n-1}}])$ such that $\rho = \tau \circ \hat\rho$ and $\sigma = \tau
\circ \hat\rho^*$. We use this fact to reorganize the sum, noting that the union of the sets $[\![(\tau \circ \rho)^* ]\!]$ and $[\![(\tau \circ \rho^*)^* ]\!]$ is $[0\range{n}]$ and their intersection is $[\![\tau^* ]\!]$:
$$
\begin{aligned}
&\mathring{\star}_T \mathring{\star}_T \omega
\\
&=
(n!|T|)^2 \sum_{\substack{\tau \in \Sigma([0\range{n-1}],[0\range{n}]) \\ \hat\rho \in \Sigma([0\range{{n-k-1}}],[0\range{{n-1}}])}}
\lambda_{(\tau \circ \hat\rho)^*}\lambda_{(\tau \circ \hat\rho^*)^*}\ 
{\star(\omega \wedge (\mathrm{d}\lambda)_{\tau \circ \hat\rho})}\ 
{\star((\mathrm{d}\lambda)_{\tau \circ \hat\rho} \wedge (\mathrm{d}\lambda)_{\tau \circ \hat\rho^*})}\ (\mathrm{d}\lambda)_{\tau\circ\hat\rho^*}.
\\
&=
\lambda_T \sum_{\tau \in \Sigma([0\range{n-1}],[0\range{n}])}
\lambda_{\tau^*}
\left\{
(n!|T|)^2
\sum_{\hat\rho \in \Sigma([0\range{{n-k-1}}],[0\range{{n-1}}])}
{\star(\omega \wedge (\mathrm{d}\lambda)_{\tau \circ \hat\rho})}\ 
{\star((\mathrm{d}\lambda)_{\tau \circ \hat\rho} \wedge (\mathrm{d}\lambda)_{\tau \circ \hat\rho^*})}\ (\mathrm{d}\lambda)_{\tau\circ\hat\rho^*}
\right\}.
\end{aligned}
$$
We show that the term in braces is $(-1)^{k(n-k)} \omega$ for each $\tau$.

Let us again number the canonical basis of $\mathbb{R}^n$ starting from zero,
$e_0, \dots, e_{n-1}$,
and let $T_{\mathrm{unit}}$ be the unit right simplex in the positive orthant, $T_{\mathrm{unit}} := \{x \in \mathbb{R}^n: x_i \geq 0, \sum_{i=0}^{n-1} x_i \leq 1\}$. For each $\tau \in \Sigma([0\range{n-1}],[0\range{n}])$, define the affine bijection $\phi_{\tau}: T \stackrel{\sim\,}{\to}T_{\mathrm{unit}}$ such that the $\tau(i)$th vertex maps to $e_i$ for $i\in[0\range{n-1}]$ and the remaining vertex maps to the
origin.  We note that because the volume of $T_{\mathrm{unit}}$ is $1/n!$, the volume pullback is $(D\phi_\tau)^* \mathrm{vol}= \pm (1 / (n! |T|))\mathrm{vol}$.
We also note that under $\phi_\tau$ the barycentric $k$-forms are pullbacks of coordinate $k$-forms if their indices are subsets of $\tau$: $(\mathrm{d}\lambda)_{\tau \circ \hat\rho} = (D\phi_{\tau})^* (\mathrm{d}x)_{\hat\rho}$ for each $\hat\rho\in \Sigma([0\range{{n-k-1}}],[0\range{{n-1}}])$.
Last, we note that the Hodge star maps a coordinate $k$-forms $(\mathrm{d}x)_{\hat\rho}$ to the coordinate $(n-k)$-forms with complementary indices, $\star(\mathrm{d}x)_{\hat\rho}
=\pm (\mathrm{d}x)_{\hat\rho^*}$.
Thus for each $\tau\in\Sigma([0\range{n-k-1}],[0\range{n}])$, the
term in braces above becomes
$$
\begin{aligned}
&(n!|T|)^2
\sum_{\hat\rho \in \Sigma([0\range{{n-k-1}}],[0\range{{n-1}}])}
{\star(\omega \wedge (\mathrm{d}\lambda)_{\tau \circ \hat\rho})}\ 
{\star((\mathrm{d}\lambda)_{\tau \circ \hat\rho} \wedge (\mathrm{d}\lambda)_{\tau \circ \hat\rho^*})}\ (\mathrm{d}\lambda)_{\tau\circ\hat\rho^*}
\\
={}
&(-1)^{k(n-k)}
(n!|T|)^2
\sum_{\hat\rho \in \Sigma([0\range{{n-k-1}}],[0\range{{n-1}}])}
{\star(\omega \wedge (\mathrm{d}\lambda)_{\tau \circ \hat\rho})}\ 
{\star((\mathrm{d}\lambda)_{\tau \circ \hat\rho^*} \wedge (\mathrm{d}\lambda)_{\tau \circ \hat\rho})}\ (\mathrm{d}\lambda)_{\tau\circ\hat\rho^*}
\\
={}
&(-1)^{k(n-k)}
(D\phi_\tau)^*
\sum_{\hat\rho \in \Sigma([0\range{{n-k-1}}],[0\range{{n-1}}])}
{\star((D\phi_\tau)^{-*}\omega \wedge (\mathrm{d}x)_{\hat\rho})}\ 
{\star((\mathrm{d}x)_{\hat\rho^*} \wedge (\mathrm{d}x)_{\hat\rho})}\ (\mathrm{d}x)_{\hat\rho^*}
\\
={}
&(-1)^{k(n-k)}
(D\phi_\tau)^*
\sum_{\hat\rho \in \Sigma([0\range{{n-k-1}}],[0\range{{n-1}}])}
\langle (D\phi_\tau)^{-*}\omega, (\mathrm{d}x)_{\hat\rho^*}\rangle 
\langle (\mathrm{d}x)_{\hat\rho^*} , (\mathrm{d}x)_{\hat\rho^*}\rangle (\mathrm{d}x)_{\hat\rho^*}
\\
={}
&(-1)^{k(n-k)}
(D\phi_\tau)^* (D\phi_\tau)^{-*} \omega
\\
={}
&(-1)^{k(n-k)} \omega.
\end{aligned}
$$
Summing these contributions over all $\tau$, we get
$$
\mathring{\star}_T \mathring{\star}_T \omega
=
(-1)^{k(n-k)}\lambda_T \omega \sum_{\tau \in \Sigma([0\range{n-1}],[0\range{n}])}
\lambda_{\tau^*}
= (-1)^{k(n-k)}\lambda_T \omega.
$$
\end{proof}

We are now ready to prove the main results of this section, that
\(\mathring{\star}_T\) implements the isomorphisms \eqref{eq:iso1} and
\eqref{eq:iso2}.

\begin{theorem}\label{thm:iso1}
$\mathring{\star}_T:\mathcal{P}_r \Lambda^k (T) \stackrel{\sim\,}{\to}\mathring{\mathcal{P}}_{r + k + 1}^-\Lambda^{n-k}(T).$
\end{theorem}

\begin{proof}
Applying $\mathring{\star}_T$ to $(\mathrm{d}\lambda)_\sigma$ for some $\sigma \in \Sigma([0\range{k-1}],[0\range{n}])$,
each term in \eqref{eq:ringo} where $[\![\rho ]\!] \cap [\![\sigma ]\!] \neq \emptyset$ vanishes.
The only nonzero summands are for $\rho$ where $[\![\rho ]\!] = [\![\sigma^* ]\!]\setminus\{\sigma^*(j)\}$ for some $j$. Hence
$$
\begin{aligned}
  \mathring{\star}_T (\mathrm{d}\lambda)_{\sigma}
    &=
     n! |T|
      \sum_{j=0}^{n-k-1}
        \star((\mathrm{d}\lambda)_{\sigma} \wedge (\mathrm{d}\lambda)_{\sigma^* \setminus \sigma^*(j)})
        \lambda_{\sigma}\lambda_{\sigma^*(j)}
        (\mathrm{d}\lambda)_{\sigma^*\setminus \sigma^*(j)}
    \\
    &=
     (-1)^k
    (\mathop{\mathrm{sign}}\sigma)
      \lambda_{\sigma}
      \sum_{j=0}^{n-k-1}
        (-1)^j
        \lambda_{\sigma^*(j)}
        (\mathrm{d}\lambda)_{\sigma^*\setminus \sigma^*(j)}
    \\
    &=
      (-1)^k
    (\mathop{\mathrm{sign}}\sigma)
      \lambda_{\sigma}
      \phi_{\sigma^*}.
\end{aligned}
$$
As a result we conclude $\mathring{\star}_T (\mathrm{d}\lambda)_{\sigma} \in \mathring{\mathcal{P}}_{k+1}^- \Lambda^{n-k}(T)$. Each polynomial in $\mathcal{P}_{r}\Lambda^k(T)$ has a representation of the form
$$
\sum_{\sigma\in \Sigma([0\range{k-1}],[0\range{n}])} a_{\sigma}(\mathrm{d}\lambda)_{\sigma},
\quad a_{\sigma} \in \mathcal{P}_r(T),
$$
so $\mathring{\star}_T$ maps $\mathcal{P}_{r} \Lambda^k(T)$ into
$\mathring{\mathcal{P}}_{r+k+1}^- \Lambda^{n-k}(T)$.
By lemma~\ref{lem:injection}, $\mathring{\star}_T$ is injective,
and it has already been established that the spaces have the same dimension,
so the operator is an isomorphism.
\end{proof}

\begin{theorem}\label{thm:iso2}
$\mathring{\star}_T:\mathcal{P}_r^- \Lambda^k (T) \stackrel{\sim\,}{\to}\mathring{\mathcal{P}}_{r + k}\Lambda^{n-k}(T).$
\end{theorem}

\begin{proof}
Let $\sigma\in \Sigma([0\range{k}],[0\range{n}])$ be given.
From the proof of theorem~\ref{thm:iso1}, we have
$$
\phi_{\sigma} = (-1)^{n-k}\frac{\mathop{\mathrm{sign}}\sigma^*}{\lambda_{\sigma^*}} \mathring{\star}_T (\mathrm{d}\lambda)_{\sigma^*}.
$$
Applying $\mathring{\star}_T$ to both sides, we have
$$
\begin{aligned}
\mathring{\star}_T \phi_{\sigma} &= (-1)^{n-k}\frac{\mathop{\mathrm{sign}}\sigma^*}{\lambda_{\sigma^*}} \mathring{\star}_T \mathring{\star}_T (\mathrm{d}\lambda)_{\sigma^*} \\
&= (-1)^{(k+1)(n-k)}\frac{(\mathop{\mathrm{sign}}\sigma^*)\lambda_T}{\lambda_{\sigma^*}} (\mathrm{d}\lambda)_{\sigma^*} \\
&= (\mathop{\mathrm{sign}}\sigma)\lambda_{\sigma} (\mathrm{d}\lambda)_{\sigma^*},
\end{aligned}
$$
which is in $\mathcal{P}_{k+1}\Lambda^{n-k}(T)$.
Therefore, as every function in $\mathcal{P}_r^- \Lambda^k(T)$ has a representation
of the form
$$
\sum_{\sigma\in \Sigma([0\range{k}],[0\range{n}])} a_{\sigma} \phi_\sigma,
$$
where $a_{\sigma} \in \mathcal{P}_{r-1} (T)$,
$\mathring{\star}_T$ maps $\mathcal{P}_r^- \Lambda^k(T)$ into
$\mathring{\mathcal{P}}_{r+k}\Lambda^{n-k}(T)$ injectively, which concludes the proof.
\end{proof}

We note that the proofs of \cref{thm:iso1,thm:iso2} show that
\(\mathring{\star}_T\) differs from \(h_T^k\) and \(h_T^{k,-}\) only by
sign conventions.

To complete the claim that \(\mathring{\star}_T\) induces an isomorphism
between \(\Lambda^k(\overline T)\) and
\(\mathring{\Lambda}^{n-k}(\overline T)\), we need to prove that
\(\mathring{\star}_T:\Lambda^k(\overline T) \to \mathring{\Lambda}^{n-k}(\overline T)\)
is a surjection, for which we require one additional lemma.

Let us define the space of \(k\)-forms that are not only trace-free but
vanish at the boundary, \[
\mathring{C}^\infty(\overline T; \mathrm{Alt}^{n-k}\mathbb{R}^n) := \{\omega \in \Lambda^{n-k}(\overline T):
\omega |_{\partial T} = 0\}.
\] It turns out that \(\mathring{\star}_T\) maps
\(\mathring{\Lambda}^k(\overline T)\) into this space.

\begin{lemma}\label{lem:czero}
$\mathring{\star}_T : \mathring{\Lambda}^k(\overline T) \to \mathring{C}^\infty(\overline T; \mathrm{Alt}^{n-k}\mathbb{R}^n)$.
\end{lemma}

\begin{proof}
  Let $x \in f_\sigma \neq T$ be given.  If $\omega \in
  \mathring{\Lambda}^k(\overline T)$, then $(\mathrm{Tr}_\sigma \omega)_x = (i_\sigma^* \omega)_x = (Di_\sigma)^* \omega_x = 0$. We use the
  fact that the nullspace of $(Di_\sigma)^*$ is spanned by
  barycentric $k$-forms whose indices are not contained in $\sigma$,
  $\{(\mathrm{d}
  \lambda)_{\rho} \in \Sigma([0\range{k-1}],[0\range{n}]): [\![\rho ]\!] \not\subseteq [\![\sigma ]\!] \}$, to say that $\omega_x$ is a linear combination of these
  $k$-forms.  For such $(\mathrm{d}\lambda)_{\rho}$, the proof of
  theorem~\ref{thm:iso1} shows that
  $$
  (\mathring{\star}_T (\mathrm{d}\lambda)_{\rho})_x = (-1)^k(\mathop{\mathrm{sign}}\rho) (\lambda_{\rho} (\mathrm{d}\lambda)_{\rho^*})_x = 0,
  $$
  a conclusion we reach because $(\lambda_{\rho})_x = 0$ for $x\in
  f_\sigma$ and $[\![\rho ]\!]\not\subseteq [\![\sigma ]\!]$.
  Therefore $(\mathring{\star}_T \omega)_x = 0$ as well.
\end{proof}

We are now able to present our final theorem of this section.

\begin{theorem}\label{thm:ringoall}
$\mathring{\star}_T: \Lambda^k(\overline T) \stackrel{\sim\,}{\to}
\mathring{\Lambda}^k(\overline T).$
\end{theorem}

\begin{proof}
Lemma~\ref{lem:injection} already shows that the map is an injection: it remains to show that it is a surjection.

Let $\omega\in\Lambda^k(\overline T)$ be given. By lemma~\ref{lem:ringotwice}, $\mathring{\star}_T \mathring{\star}_T \omega = (-1)^{k(n-k)}\lambda_T \omega$, so $\mathring{\star}_T \circ \mathring{\star}_T$
is a bijection,
$\mathring{\star}_T \circ \mathring{\star}_T: \Lambda^k(\overline T) \stackrel{\sim\,}{\to}\mathring{C}^\infty(\overline T; \mathrm{Alt}^k\mathbb{R}^n)$.
Given $\mu \in \mathring{\Lambda}^{n-k}(\overline T)$, by lemma~\ref{lem:czero} we can define
an element $\omega := (\mathring{\star}_T \circ \mathring{\star}_T)^{-1} \mathring{\star}_T \mu\in \Lambda^k(\overline T)$.
$\mathring{\star}_T \omega$ is in $\mathring{\Lambda}^{n-k}(\overline T)$, and
$\mathring{\star}_T \mathring{\star}_T \omega = \mathring{\star}_T \mu$.  By the injectivity of $\mathring{\star}_T$, $\mathring{\star}_T \omega = \mu$.
\end{proof}

\hypertarget{optimal-dual-basis-functionals-via-the-mathringstar_t-inner-product}{%
\subsection{\texorpdfstring{Optimal dual basis functionals via the
\(\mathring{\star}_T\) inner
product}{Optimal dual basis functionals via the \textbackslash{}mathring\{\textbackslash{}star\}\_T inner product}}\label{optimal-dual-basis-functionals-via-the-mathringstar_t-inner-product}}

Because the operator \(\mathring{\star}_T\) is defined pointwise, it can
be used in practice to adapt an unrestricted basis with desirable
qualities -- some combination of numerical stability and computational
efficiency -- into a basis for trace-free subspaces.

Implementations of the finite element method may also define the basis
functions indirectly, opting instead to define a unisolvent set of
functionals and constructing the basis functions by inverting the
generalized vandermonde matrix of a numerically stable basis for the
primal space. The canonical basis dual basis functionals for the trace
free subspace \(\mathring{\mathcal{P}}_r \Lambda^k(T)\) given in
\cite[Sections 4.5]{arnold2006finite} are \[
\phi_i(\omega) = \int_T \omega \wedge \eta_i,
\] where \(\{\eta_i\}\) is a basis of
\(\mathcal{P}_{r - (n - k)}^- \Lambda^k(T)\); for
\(\mathring{\mathcal{P}}_r^- \Lambda^k(T)\), they are the basis
functionals given in \cite[Sections 4.6]{arnold2006finite}, \[
\phi_i(\omega) = \int_T \omega \wedge \eta_i,
\] where \(\{\eta_i\}\) is a basis of
\(\mathcal{P}_{r - (n - k) - 1} \Lambda^k(T)\).

That work does not suggest which dual basis differential forms
\(\{\eta_i\}\) to use in either case. The inner product
\(\langle \cdot, \cdot \rangle_{\mathring{\star}_T}\) defined in
\eqref{eq:ringot-inner-product} in the proof of
lemma\textasciitilde{}\ref{lem:injection} is useful for this purpose. If
a basis \(\{\eta_i\}\) is chosen that is orthonormal with respect to
this inner product, then the corresponding basis functions can be
computed directly as \(\{\mathring{\star}_T \eta_i\}\). For scalar
polynomials, such orthonormal polynomials can be evaluated explicitly
using Dubiner-type combinations of Gauss-Jacobi polynomials, such as in
the basis of \textcite{sherwin1995new}. We do not investigate the
construction of an explicit orthonormal basis for
\(\langle \cdot, \cdot \rangle_{\mathring{\star}_T}\) further in this
work.

\hypertarget{extension-operators}{%
\section{Extension operators}\label{extension-operators}}

\label{sec:extension}

\hypertarget{consistent-families-of-extension-operators}{%
\subsection{Consistent families of extension
operators}\label{consistent-families-of-extension-operators}}

In \cite{arnold2009geometric}, the authors define a \emph{consistent
family of extension operators} for geometric decomposition. Letting
\(X\) stand in for \(\mathcal{P}_r \Lambda^k\) or
\(\mathcal{P}_r^- \Lambda^k\), a consistent family of extension
operators is a set of operators
\(\{E_{f,g}: X(f) \to X(g), f,g\in\mathcal{T}_h \}\) for any pair of
simplices \(f\) and \(g\) in a mesh \(\mathcal{T}_h\), for which trace
and extension operators commute in the following sense: if \(f\) and
\(g\) are boundary simplices of simplex \(h\), and \(f\cap g\) is the
intersection of their boundaries, then
\begin{equation}\label{eq:consistent}
\mathrm{Tr}_{h,g} E_{f,h} \omega = E_{f \cap g,g}\mathrm{Tr}_{f,f\cap g} \omega
\quad \omega \in X(f).
\end{equation} This property guarantees that the extension of a
trace-free \(k\)-form \(\omega\in X(f_\sigma)\) into neighboring cells
has trace continuity at the boundaries between the cells, and so it can
serve as a local basis function in an \(H\Lambda^k(\Omega\))-conforming
finite element space.

There are many ways to define consistent families of extension
operators, but in this work we only define extension operators that are
affine invariant and respect all of the symmetries of the simplex. This
allows us to give our definitions with respect to the single-element
mesh made up of the simplex \(T\) and its boundary simplices, and the
results can be mapped to a general mesh \(\mathcal{T}_h\) with
properties preserved. We also note that, for the purpose of geometric
decomposition, it is not necessary to design an extension operator whose
domain is all of \(X(f)\), because the only \(k\)-forms that are
extended are in \(\mathring{X}(f)\).

Using these facts, in this section we design a family of extension
operators
\(\{\dot{E}_{\rho,\sigma}: \mathring{X}(f_\rho) \to X(f_\sigma), [\![\rho ]\!], [\![\sigma ]\!] \subseteq [0\range{n}]\}\)
such that for any \(f_\tau\) such that
\([\![\rho ]\!],[\![\sigma ]\!]\subset[\![\tau ]\!]\), \[
\mathrm{Tr}_{\tau,\sigma} \dot{E}_{\rho,\tau} \omega = \dot{E}_{\rho\cap\sigma,\sigma} \mathrm{Tr}_{\rho,\rho\cap\sigma} \omega, \quad \omega \in X(f_\rho).
\] This is sufficient to show that one has the geometric decomposition
\[
X(T) = \bigoplus_{\sigma} \dot{E}_{\sigma,T}[\mathring{X}(f_\sigma)],
\] where we can take not only \(X(f) = \mathcal{P}_r\Lambda^k(f)\) or
\(X(f) = \mathcal{P}_r^-\Lambda^k(f)\), but even
\(X = \Lambda^k(\overline f)\).

\hypertarget{additional-notation}{%
\subsection{Additional notation}\label{additional-notation}}

We identify \(k\)-forms whose domain if \(f_\sigma\) with a superscript,
such as \(\lambda^{(\sigma)}\), \(\mathrm{d}\lambda^{(\sigma)}\),
\(\phi_\rho^{(\tau)}\).

\paragraph{Centroid projectors.}

For each \(d\)-simplex \(f_{\sigma}\) in the boundary of \(T\), the
\emph{centroid projector} \(P_{T,\sigma}: T \to f_\sigma\) is an affine
map defined in terms of its actions on the vertices \(\{v_i\}_{i=0}^n\)
of \(T\): if \(i \in \sigma\), then \(P_{T,\sigma} (v_i) = v_i\),
otherwise \(v_i\) is mapped to the centroid of \(f_{\sigma}\),
\(P_{T,\sigma} (v_i) = \frac{1}{d}\sum_{i = 0}^d v_{\sigma(i)}\). Given
two boundary simplices \(f_{\sigma}\) and \(f_{\rho}\), with
\([\![\rho ]\!] \subset [\![\sigma ]\!]\), then we define the centroid
projector \(P_{\sigma,\rho}: f_{\sigma} \to f_{\rho}\) analogously.
Centroid projectors compose with each other, in that if
\([\![\rho ]\!] \subseteq [\![\sigma ]\!] \subseteq [\![\tau ]\!]\),
then \begin{equation}\label{eq:p-compose}
P_{\sigma,\rho} \circ P_{\tau, \sigma} = P_{\tau,\rho}.
\end{equation} The centroid project \(P_{\sigma,\rho}\) is also a left
inverse of the inclusion map
\(i_{\rho,\sigma}:f_{\rho} \hookrightarrow f_{\sigma}\), which implies
that \(\mathrm{Tr}_{\sigma,\rho}\) is a left inverse for
\(P_{\sigma,\rho}^*\). Combining this fact with \eqref{eq:p-compose}, we
get \begin{equation}\label{eq:tr-inv-p}
P_{\sigma,\rho}^* = \mathrm{Tr}_{\tau,\sigma} \circ P_{\tau,\rho}^*.
\end{equation}

\paragraph{Centroid Koszul operators.}

For each simplex \(f_\sigma\), we let \(\kappa^{(\sigma)}\) be the
Koszul operator centered at the centroid of \(f_\sigma\). Because the
centroid projector \(P_{\sigma,\rho}\) maps the centroid of \(f_\sigma\)
to the centroid of \(f_\rho\), \eqref{eq:kozsul-commute} implies \[
\kappa^{(\sigma)} P_{\sigma,\rho}^* = P_{\sigma,\rho}^* \kappa^{(\rho)}.
\]

\hypertarget{previous-extension-operators}{%
\subsection{Previous extension
operators}\label{previous-extension-operators}}

In \cite{arnold2009geometric}, the authors define two extension
operators,
\(E_{\sigma,T}^{r,k}: \mathcal{P}_r \Lambda^k(f_{\sigma}) \to \mathcal{P}_r \Lambda^k(T)\)
and
\(E_{\sigma,T}^{r,k,-}: \mathcal{P}_r^- \Lambda^k(f_{\sigma}) \to \mathcal{P}_r^- \Lambda^k(T)\),
and show that they are consistent extension operators that can be used
in geometric decompositions of their target spaces. Those extension
operators are defined as follows:

\begin{align}
E_{\sigma,T}^{r,k}: (\lambda_{\sigma}^{(\sigma)})^{\alpha} (\mathrm{d}\lambda^\sigma)_\tau &\mapsto
(\lambda_\sigma)^{\alpha} P_{T,\sigma,\alpha}^*(\mathrm{d}\lambda^{\sigma})_{\tau}, &|\alpha| = r, \tau \in \Sigma([0\range{k}], {[\![\sigma ]\!]}); \label{eq:extfull} \\
E_{\sigma,T}^{r,k,-}: (\lambda_{\sigma}^{(\sigma)})^{\alpha} \phi_{\tau}^{(\sigma)} &\mapsto (\lambda_{\sigma})^{\alpha}
\phi_{\tau}, &|\alpha| = r - 1, \tau \in
\Sigma([0\range{k-1}], {[\![\sigma ]\!]}). \label{eq:exttrimmed}
\end{align}

The projector \(P_{T,\sigma,\alpha}:T \to f_{\sigma}\) used in
\eqref{eq:extfull} differs from the centroid projector in that the
vertices that are not in \(f_\sigma\) map not to the centroid of
\(f_\sigma\), but to \(x_\alpha\), a weighted average of the vertices of
\(f_\sigma\) with weights determined by the multi-index \(\alpha\), \[
x_\alpha = \frac{1}{|\alpha|} \sum_{i=0}^{d} \alpha(i) v_{\sigma(i)}.
\]

As was the case with the trace-free maps in \cref{sec:prev-trace-free},
these extension operators are defined in terms of their action on the
products of scalar polynomials and barycentric \(k\)-forms (for full
polynomials) or Whitney forms (for trimmed). Unlike that case, though,
these operators are mutually incompatible, in that neither maps the
other into the correct space.

We demonstrate their incompatibility when extending 1-forms from the
triangle \(\sigma = (v_1,v_2,v_3)\) to the tetrahedron
\(T = (v_0,v_1,v_2,v_3)\) by finding specific examples where a 1-form
extends to the wrong space.

\textbf{Example showing
\(E_{\sigma,T}^{k,r,-} \mathring{\mathcal{P}}_{r-1}\Lambda^k(f_\sigma) \not\to \mathcal{P}_{r-1}\Lambda^k(T)\).}
In this case \(k=1\) and \(r=3\). Let
\(\omega = \lambda_1^{(\sigma)} \lambda_2^{(\sigma)} \mathrm{d}\lambda_3^{(\sigma)} \in \mathring{\mathcal{P}}_2 \Lambda^1(f_\sigma)\).

We can expand \(\omega\) in a Whitney form basis as
\(\omega = \lambda_1^{(\sigma)} \lambda_2^{(\sigma)} (\phi_{23}^{(\sigma)} + \phi_{13}^{(\sigma)})\),
so applying \(E_{\sigma,T}^{3,1,-}\) by the definition in
\eqref{eq:exttrimmed} yields

\[
E_{\sigma,T}^{3,1,-} \omega =
E_{\sigma,T}^{3,1,-} (
\lambda_1^{(\sigma)} \lambda_2^{(\sigma)} (\phi_{23}^{(\sigma)} + \phi_{13}^{(\sigma)})
)
=
\lambda_1 \lambda_2 (\phi_{23} + \phi_{13}).
\]

On the tetrahedron \(\phi_{23} + \phi_{13} \neq \mathrm{d}\lambda_3\),
so \(E_{\sigma,T}^{3,1,-}\omega \not\in\mathcal{P}_2 \Lambda^1(T).\)

\textbf{Example showing
\(E_{\sigma,T}^{k,r} \mathring{\mathcal{P}}_r^- \Lambda^k(f_\sigma) \not\to \mathcal{P}_r^- \Lambda^k(T)\).}
Let
\(\omega = \lambda_1^{(\sigma)}\lambda_2^{(\sigma)}\phi_{23}^{(\sigma)} \in \mathring{\mathcal{P}}_3^- \Lambda^1(f_\sigma)\)
(in this case again \(k=1\) and \(r=3\)). Expanded in a barycentric
1-form basis, \(\omega\) is
\(\lambda_1^{(\sigma)}(\lambda_2^{(\sigma)})^2\mathrm{d}\lambda_3^{(\sigma)} - \lambda_1^{(\sigma)} \lambda_2^{(\sigma)} \lambda_3^{(\sigma)} \mathrm{d}\lambda_2^{(\sigma)}.\)
According to \eqref{eq:extfull}, \(E_{\sigma,T}^{3,1}\) extends the two
1-forms in this representation of \(\omega\) by the pullbacks of two
different projections, \(P_{T,\sigma,(1,2,0)}\) and
\(P_{T,\sigma,(1,1,1)}\), determined by their barycentric monomials. The
first projection \(P_{T,\sigma,(1,2,0)}\) can be described in terms of
barycentric coordinates as \[(\lambda_1^{(\sigma)},
\lambda_2^{(\sigma)},
\lambda_3^{(\sigma)}) \gets (\lambda_1 + (1/3)\lambda_0,
\lambda_2 + (2/3)\lambda_0, \lambda_3),\] so
\(P_{T,\sigma,(1,2,0)}^*\mathrm{d}\lambda_3^{(\sigma)} = \mathrm{d}\lambda_3.\)
The second projection \(P_{T,\sigma,(1,1,1)}\) is \[
(\lambda_1^{(\sigma)},
\lambda_2^{(\sigma)},
\lambda_3^{(\sigma)}) \gets (\lambda_1 + (1/3)\lambda_0,
\lambda_2 + (1/3)\lambda_0, \lambda_3 + (1/3) \lambda_0),
\] so
\(P_{T,\sigma,(1,1,1)}^*\mathrm{d}\lambda_2^{(\sigma)} = \mathrm{d}\lambda_2 + (1/3) \mathrm{d}\lambda_0.\)
All together, this shows that the extension of \(\omega\) is \[
E_{T,\sigma}^{3,1}\omega =
\lambda_1 \lambda_2^2 \mathrm{d}\lambda_3 - 
\lambda_1 \lambda_2 \lambda_3 (\mathrm{d}\lambda_2 + (1/3)\mathrm{d}\lambda_0).
\] The easiest way to show that this is not in
\(\mathcal{P}_3^- \Lambda^1(T)\) is to apply the Koszul operator because
\(\kappa_x p \in \mathcal{P}_3 \Lambda^0(T)\) for every
\(p \in \mathcal{P}_3^- \Lambda^1(T)\) and every \(x\). By
\cite[Theorem 3.1]{arnold2006finite} we have
\(\kappa_x(\lambda^\alpha \mathrm{d}\lambda_i) = \lambda^\alpha (\lambda_i - \lambda_i(x))\)
for each multi-index \(\alpha\) and each \(i \in [0\range{n}]\).
Choosing \(x = v_1\) where
\(\lambda_0(v_1) = \lambda_2(v_i) = \lambda_3(v_1) = 0\) gives us \[
\kappa_{v_1} E_{T,\sigma}^{3,1}\omega = -(1/3) \lambda_0\lambda_1\lambda_2\lambda_3 \not\in \mathcal{P}_3 \Lambda^0(T).
\]

\hypertarget{the-bubble-decomposition-of-mathringlambdakoverline-t}{%
\subsection{\texorpdfstring{The bubble decomposition of
\(\mathring{\Lambda}^k(\overline T)\)}{The bubble decomposition of \textbackslash{}mathring\{\textbackslash{}Lambda\}\^{}k(\textbackslash{}overline T)}}\label{the-bubble-decomposition-of-mathringlambdakoverline-t}}

Even though each \(\omega \in \mathring{\Lambda}^k(\overline T)\) is
trace-free, \(\omega_x\) for \(x\in \partial T\) is not necessarily zero
because there can be components of \(\omega_x\) that are perpendicular
to the trace operator \(\mathrm{Tr}\) at \(x\). In this section we
define a way to decompose a trace-free \(k\)-form based on these
perpendicular traces called the \emph{bubble decomposition}. This
decomposition is the foundation of the unified extension operator we
define below. We present the bubble decomposition for the \(n\)-simplex
\(T\): the definition extends naturally to
\(\mathring{\Lambda}^k(\overline{f_\sigma})\) for each boundary simplex
\(f_\sigma\).

Let \(k\in[0\range{n}]\) be given. For each \(d \geq n - k\) and each
\(\sigma \in \Sigma([0\range{d}],[0\range{n}])\), let \(\omega_\sigma\)
be a \((k-(n-d))\)-form in \(\Lambda^{k - (n-d)}(\overline{f_\sigma})\).
Define
\(\mathring{E}_{\sigma,T}: \Lambda^{k-(n-d)} (\overline{f_\sigma}) \to \mathring\Lambda^k(\overline{T})\)
by

\begin{equation}\label{eq:ringe}
\mathring{E}_{\sigma,T}: \omega_\sigma \mapsto P_{T,\sigma}^*\omega_\sigma \wedge (\lambda_{\sigma}(\mathrm{d}\lambda)_{\sigma^*}).
\end{equation}

The range of \(\mathring{E}_{\sigma,T}\) is trace-free because the form
\(\lambda_\sigma(\mathrm{d}\lambda)_{\sigma^*}\) appearing in the right
hand side of \eqref{eq:ringe} is trace-free and because trace operators
distribute over the wedge product. We first show that
\(\mathring{E}_{\sigma,T}\) is an injection.

\begin{lemma}\label{lem:ringe-injective}
$\mathring{E}_{\sigma,T}$ defined by \eqref{eq:ringe}
is injective.
\end{lemma}

\begin{proof}
Let $\hat\Sigma = \Sigma([1\range{k-(n-d)}],[\![\sigma\setminus \sigma(0) ]\!])$.
We note that $\{(\mathrm{d}\lambda^{(\sigma)})_{\hat\rho}\}_{\hat\rho\in\hat\Sigma}$ is a basis for $\mathrm{Alt}^{k-(n-d)} T_x f_\sigma$,
so it is sufficient to show that
$\mathring{E}_{\sigma,T} \alpha (\mathrm{d}\lambda^{(\sigma)})_{\hat\rho} \neq 0$
for every nonzero $\alpha \in C^\infty(\overline{f_\sigma})$ and every $\hat\rho\in\hat\Sigma$.

We note that the pullback of a barycentric $1$-form by $(DP_{T,\sigma})^*$ is
$$
(DP_{T,\sigma})^* \mathrm{d}\lambda_{i}^{(\sigma)} = 
\mathrm{d}\lambda_i + \frac{1}{d+1} \sum_{j\in [\![\sigma^* ]\!]} \mathrm{d}\lambda_j, \quad i \in [\![\sigma ]\!].
$$
Because $\mathrm{d}\lambda_j \wedge \mathrm{d}\lambda_{\sigma^*} = 0$ if and only if
$j\in [\![\sigma^* ]\!]$, this implies
$(DP_{T,\sigma})^* \mathrm{d}\lambda_i^{(\sigma)} \wedge \mathrm{d}\lambda_{\sigma^*}
=
\mathrm{d}\lambda_i \wedge \mathrm{d}\lambda_{\sigma^*}$
for each $i \in [\![\sigma ]\!]$, and in general
$$
(DP_{T,\sigma})^* (\mathrm{d}\lambda^{(\sigma)})_{\hat\rho} \wedge \mathrm{d}\lambda_{\sigma^*}
=
(\mathrm{d}\lambda)_{\hat\rho} \wedge \mathrm{d}\lambda_{\sigma^*}, \quad
[\![\hat\rho ]\!] \subseteq [\![\sigma ]\!].
$$
Therefore
$$
\begin{aligned}
\mathring{E}_{\sigma,T} \alpha (\mathrm{d}\lambda^{(\sigma)})_{\hat\rho}
&=
P_{T,\sigma}^*(\alpha (\mathrm{d}\lambda^{(\sigma)})_{\hat\rho}) 
\wedge (\lambda_\sigma (\mathrm{d}\lambda)_{\sigma^*})
\\
&=
(\lambda_\sigma P_{T,\sigma}^* \alpha)\ ((D P_{T,\sigma})^* (\mathrm{d}\lambda^{(\sigma)})_{\hat\rho}
\wedge (\mathrm{d}\lambda)_{\sigma^*})
\\
&=
(\lambda_\sigma P_{T,\sigma}^* \alpha)\ ((\mathrm{d}\lambda)_{\hat\rho}
\wedge (\mathrm{d}\lambda)_{\sigma^*}).
\end{aligned}
$$
Neither the scalar term nor the algebraic $k$-form term is identically zero, so
$\mathring{E}_{\sigma,T} \alpha (\mathrm{d}\lambda^{(\sigma)})_{\hat\rho}
\neq 0$.
\end{proof}

The following lemma shows how \(\mathring{E}_{\sigma,T}\) can represent
the components of \(\omega_x\) perpendicular to \(\mathrm{Tr}_{\sigma}\)
for \(x\in f_\sigma\) and
\(\omega\in\mathring{\Lambda}^k(\overline{T})\).

\begin{lemma}\label{lem:ringe-inverse}
Let $\omega \in \mathring{\Lambda}^k(\overline{T})$ be given and let
a $d$-simplex $f_\sigma$ be given such that $d \geq n - k$.
Suppose that $\omega$ vanishes at the boundaries of $f_\sigma$, that is
$\omega|_{f_\tau} = 0$ for every $\tau$ such that $[\![\tau ]\!] \subsetneq [\![\sigma ]\!]$.  Then there is a unique $\omega_\sigma \in \Lambda^{k - (n-d)}(\overline{f_\sigma})$ such that
$$
\omega|_{f_\sigma} = \mathring{E}_{\sigma,T}(\omega_\sigma) |_{f_\sigma}.
$$
\end{lemma}

\begin{proof}
By \cref{thm:ringoall},
there exists $\mu \in \Lambda^{n-k}(\overline{T})$ such that
$\omega = \mathring{\star}_T \mu$, so for $x\in f_\sigma$,

$$
\begin{aligned}
\omega_x &=
n!|T| \sum_{\rho\in\Sigma([0\range{k-1}],[0\range{n}])}
\star(\mu_x \wedge (\mathrm{d}\lambda)_\rho)\ (\lambda_{\rho^*})_x (\mathrm{d}\lambda)_{\rho}
\\
&=
n!|T| \sum_{\substack{\rho\in\Sigma([0\range{k-1}],[0\range{n}]) \\ [\![\rho^* ]\!] \subseteq [\![\sigma ]\!]}} \star(\mu_x \wedge (\mathrm{d}\lambda)_\rho)\ (\lambda_{\rho^*})_x (\mathrm{d}\lambda)_{\rho},
\end{aligned}
$$
where we have eliminated terms from the sum that are zero because
$\lambda_{\rho^*}|_{f_\sigma} = 0.$
But if $\rho \in \Sigma([0\range{k-1}],[0\range{n}])$ and $[\![\rho^* ]\!] \subseteq
[\![\sigma ]\!]$, then there is $\hat\rho \in \Sigma([0\range{k-(n-d)-1}],[\![\sigma ]\!])$
such that
$(\mathrm{d}\lambda)_{\rho} = \pm (\mathrm{d}\lambda)_{\hat\rho} \wedge (\mathrm{d}\lambda)_{\sigma^*}.$
Therefore there exist some smooth functions $\{\alpha_{\hat\rho}\}
\subset C^\infty (\overline{f_\sigma})$
such that
$$
\omega|_{f_\sigma} =
\sum_{\hat\rho\in\Sigma([0\range{k-(n-d)-1}],[\![\sigma ]\!])} \alpha_{\hat\rho} ((\mathrm{d}\lambda)_{\hat\rho} \wedge (\mathrm{d}\lambda)_{\sigma^*}).
$$

By the same argument from the proof of lemma~\ref{lem:ringe-injective},
we have $(\mathrm{d}\lambda)_{\hat\rho} \wedge (\mathrm{d}\lambda)_{\sigma^*} = P_{T,\sigma}^* 
(\mathrm{d}\lambda^{(\sigma)})_{\hat\rho} \wedge (\mathrm{d}\lambda)_{\sigma^*}$,
so
$$
\omega|_{f_\sigma} =
\sum_{\hat\rho\in\Sigma([0\range{k-(n-d)-1}],[\![\sigma ]\!])} \alpha_{\hat\rho} (P_{T,\sigma}^*(\mathrm{d}\lambda^{(\sigma)})_{\hat\rho} \wedge (\mathrm{d}\lambda)_{\sigma^*}).
$$
Finally, we stipulated that $\omega$ vanishes at the boundaries
of $f_\sigma$, so for each $\alpha_{\hat\rho}$ there is
$\tilde{\alpha}_{\hat\rho} \in C^{\infty}(\overline{f_\sigma})$ such that
$\alpha_{\hat\rho} = \lambda_{\sigma} \tilde{\alpha}_{\hat\rho}$,
and so
$$
\begin{aligned}
\omega|_{f_\sigma} &=
P_{T,\sigma}^*
\big(
\sum_{\hat\rho\in\Sigma([0\range{k-(n-d)-1}],[\![\sigma ]\!])} \tilde\alpha_{\hat\rho} (\mathrm{d}\lambda^{(\sigma)})_{\hat\rho} \big)\wedge (\lambda_{\sigma}(\mathrm{d}\lambda)_{\sigma^*})
\\
&=
\mathring{E}_{\sigma,T} \big(
\sum_{\hat\rho\in\Sigma([0\range{k-(n-d)-1}],[\![\sigma ]\!])} \tilde\alpha_{\hat\rho} (\mathrm{d}\lambda^{(\sigma)})_{\hat\rho} \big)
\big|_{f_\sigma}.
\end{aligned}
$$
The uniqueness follows from lemma~\ref{lem:ringe-injective}.
\end{proof}

Having established a way to encode the perpendicular component of
\(\omega|_{f_\sigma}\) for a single boundary simplex \(f_\sigma\) using
\(\mathring{E}_{\sigma,T}\), we now construct the full bubble
decomposition of a trace-free \(k\)-form.

We denote with \(Z^k(T)\) the set of increasing maps \(\sigma\) for
which \(\mathring{E}_{\sigma,T}\) is defined, \[
Z^k(T) := \bigcup_{d = n-k}^n \Sigma([0\range{d}],[0\range{n}]),
\] and use \(\sigma\in Z^k(T)\) to index the \emph{\(k\)-bubble trace
space} \(\mathbf{B}^k(T)\), the product of the spaces for which
\(\mathring{E}_{\sigma,T}\) is defined,

\[
\mathbf{B}^k(T) :=  \bigotimes_{\sigma \in Z^k(T)}
\Lambda^{k - (n- \dim(f_\sigma))}(\overline{f_\sigma}).
\]

Finally, we define
\(\mathring{E}_T:\mathbf{B}^k(T) \to \mathring{\Lambda}^k(\overline T)\)
to be the sum of the extensions of each of the components of the product
space: given \(W = \otimes_{\sigma\in Z^k(T)} \omega_\sigma\),

\begin{equation}\label{eq:bubbleinv}
\mathring{E}_T: W \mapsto \sum_{\sigma \in Z^k(T)} \mathring{E}_{\sigma,T}\omega_\sigma.
\end{equation}

\begin{theorem}\label{thm:ringe-biject}
$\mathring{E}_T$ is a bijection,
$\mathring{E}_T: \mathbf{B}^k(T) \stackrel{\sim\,}{\to}\Lambda^k(\overline T).$
\end{theorem}

\begin{proof}
Lemma~\ref{lem:ringe-injective} showed that 
$\mathring{E}_{\sigma,T}$ is injective for each individual boundary simplex $f_\sigma$, but we must show
that the sum of contributions from different boundary simplices is still injective.

Suppose
$\omega = \mathring{E}_T W =  \sum_{\sigma \in Z^k(T)} \mathring{E}_{\sigma,T}\omega_\sigma = 0.$
Let $d$ be the smallest dimension such that $\omega_\sigma= 0$ for
every $\sigma \in Z^k(T)$ such that $\dim(f_\sigma) < d$, and choose $\hat\sigma\in Z^k(T)$ such that $\dim(f_{\hat\sigma}) = d$.
Given $\tau \in Z^k(T)$ such that $\tau \neq \sigma$ and
$\dim(f_{\tau}) \geq d$, we have
\begin{equation}\label{eq:cross-trace}
(\mathring{E}_{\tau,T} \omega_\tau)|_{f_{\hat\sigma}}
=
(P_{\tau,T}^* \omega) \wedge
(\lambda_{\tau} (\mathrm{d}\lambda)_{\tau^*})|_{f_{\hat\sigma}} = 0,
\end{equation}
a conclusion we reach because $[\![\tau ]\!] \not\subseteq [\![\hat\sigma ]\!]$, and so $\lambda_{\tau}|_{f_{\hat\sigma}} = 0.$
Therefore
$$
0 = \omega|_{f_\sigma} = (\mathring{E}_T W)|_{f_\sigma} =
(\mathring{E}_{\hat\sigma,T} \omega_{\hat\sigma})|_{f_\sigma}.
$$
By the same argument as in lemma~\ref{lem:ringe-injective}, this
implies $\omega_{\hat\sigma} = 0.$  Because $\hat\sigma$ was
arbitrarily chosen, $\omega_\sigma = 0$ for each $d$-simplex
$f_\sigma$, but this violates the way $d$ was chosen, so
we must conclude $W = 0$.

Now we prove that $\mathring{E}_{T}$ is surjective.
Suppose $\omega \in \Lambda^k(\overline T)$ vanishes at $\partial T$.
Then there exists $\tilde{\omega} \in \Lambda^k(\overline T)$ such that
$\omega = \lambda_T \tilde{\omega} = \mathring{E}_{T,T}\tilde \omega.$

Now suppose that we have shown that the range of $\mathring{E}_{T}$
includes all $k$-forms that vanish at all $f_\sigma$ with $\dim(f_\sigma) \leq d$, and let $\omega \in \Lambda^k$ be a $k$-form that vanishes at each $f_\sigma$ with $\dim(f_\sigma) < d.$
For each $\hat\sigma$ such that $\dim(f_{\hat\sigma}) = d$ there is, by lemma~\ref{lem:ringe-inverse}, $\hat\omega_{\hat\sigma}\in\Lambda^{k-(n-d)}(\overline{f_{\hat\sigma}})$ such that
$(\mathring{E}_{\hat\sigma,T} \hat\omega_{\hat\sigma})|_{f_{\hat\sigma}}
= \omega|_{f_{\hat\sigma}}.$
Let
$$
\tilde \omega = \omega - \sum_{\hat\sigma \in \Sigma([0\range{d}],[0\range{n}])}
\mathring{E}_{\hat\sigma,T} \hat \omega_{\hat\sigma}.
$$
By \eqref{eq:cross-trace}, $\tilde \omega$ vanishes at
all $d$ simplices, so there exists $\tilde{W} = \otimes_{\sigma\in Z^k(T)} \tilde{\omega}_{\sigma}$
such that $\tilde{\omega} = \mathring{E}_{T} \tilde{W}$.  Hence defining
$W = \otimes_{\sigma\in Z^k(T)} \omega_\sigma$ by
$$
\omega_{\sigma} = \begin{cases}
\tilde{\omega}_{\sigma}, & \dim(f_\sigma) \neq d, \\
\tilde{\omega}_{\sigma}  - \hat{\omega}_{\sigma}, & \dim(f_\sigma) = d,
\end{cases}
$$
we have constructed $W$ such that $\mathring{E}_{T} W = \omega$.
By induction, we conclude $\mathring{E}_{T}$ is surjective.
\end{proof}

The proof of the bijection
\(\mathring{E}_T: \mathbf{B}^k(T) \stackrel{\sim\,}{\to}\mathring{\Lambda}^k(\overline T)\)
shows that \(\mathring{E}_{\sigma,T}\) operators define a decomposition
of \(\mathring{\Lambda}(\overline T)\),
\begin{equation}\label{eq:bubble-decomp}
\mathring{\Lambda}^k(\overline T)
= \bigoplus_{\sigma \in Z^k(T)} \mathring{E}_{\sigma,T} [
\Lambda^{k - (n-d)}(\overline f_\sigma)
].
\end{equation} We define the inverse of \(\mathring{E}_T\) to be the
operator
\(\dot{B}_T^k: \mathring{\Lambda}^k(\overline T) \stackrel{\sim\,}{\to}\mathbf{B}^k(T)\),
and we let
\(\dot{B}_{T,\sigma}^k: \mathring{\Lambda}^k(\overline T) \stackrel{\sim\,}{\to}\Lambda^{k-(n-d)}(\overline{f_\sigma})\)
be the projection of \(\dot{B}_{T,\sigma}^k\) onto the \(\sigma\)
component. We have two remarks about this operator.

First, the proof of \cref{thm:ringe-biject} shows a constructive
procedure for evaluating \((B_{T,\sigma}^k \omega)_x\). One must
evaluate \(\omega\) at \(P_{T,\rho}(x)\) for each simplex \(f_\rho\)
such that \([\![\rho ]\!] \subseteq [\![\sigma ]\!]\), and solve the
system \[
\omega_{P_{T,\rho}(x)} = \lambda_\rho (B_{T,\rho}^k \omega)_{P_{T,\rho}(x)}
+ \sum_{\substack{\hat\rho\in Z^k(T) \\ [\![\hat\rho ]\!] \subset [\![\rho ]\!]}}
(\mathring{E}_{\hat\rho,T}
B_{T,\hat\rho}^k \omega)_{P_{T,\rho}(x)}, \quad
[\![\rho ]\!] \subseteq [\![\sigma ]\!].
\] This is a triangular system that can be solved from the smallest
dimension to the largest.

The second remark is that, while we have shown that this bubble
decomposition exists for all of \(\mathring{\Lambda}^k(\overline T)\),
it cannot be used to define a basis for every subspace
\(X(T) \subset \Lambda^k(\overline T)\). On the one hand, it can be used
to define a basis for the trace-free full polynomials
\(\mathring{\mathcal{P}}_r \Lambda^k(T),\)
\begin{equation}\label{eq:bubble-full}
\mathring{\mathcal{P}}_r \Lambda^k(T)
= \bigoplus_{\sigma \in Z^k(T)} \mathring{E}_{\sigma,T} [
\mathcal{P}_{r - \dim(f_\sigma)-1}\Lambda^{k - (n - \dim(f_\sigma))} (f_\sigma)
].
\end{equation}

On the other hand, the components of a trace-free trimmed \(k\)-form
\(\omega \in \mathring{\mathcal{P}}_r^- \Lambda^k(\overline T)\) will
not necessarily be in the same space. As a simple example, consider
\(\lambda_0 \phi_{12} \in \mathring{\mathcal{P}}_2^- \Lambda^1 (T)\) for
the triangle \(T = \{v_0, v_1, v_2\}\). Its bubble decomposition is \[
\lambda_0 \phi_{12} =
\lambda_0 \lambda_2 \mathrm{d}\lambda_1
-\lambda_0 \lambda_1 \mathrm{d}\lambda_2 =
\underbrace{P_{T,(0,2)}^*(1) \wedge (\lambda_0 \lambda_2 \mathrm{d}\lambda_1)}_{\mathring{E}_{(0,2),T} (1)}
+
\underbrace{P_{T,(0,1)}^*(-1) \wedge (\lambda_0 \lambda_1 \mathrm{d}\lambda_2)}_{\mathring{E}_{(0,1),T} (-1)},
\] and each of its bubble components is in
\(\mathcal{P}_2 \Lambda^1(T)\) but not \(\mathcal{P}_2^- \Lambda^1(T)\).

\hypertarget{the-dote_sigmat-operator}{%
\subsection{\texorpdfstring{The \(\dot{E}_{\sigma,T}\)
operator}{The \textbackslash{}dot\{E\}\_\{\textbackslash{}sigma,T\} operator}}\label{the-dote_sigmat-operator}}

We can generalize the bubble decomposition from the previous section to
boundary simplices. Given \(\sigma\) and \(\xi\) such that
\([\![\sigma ]\!]\subseteq[\![\xi ]\!]\), we can define
\(\mathring{E}_{\sigma,\xi}: \Lambda^{k - (\dim(f_\xi) - \dim(f_\sigma))}(\overline{f_\sigma}) \to \mathring{\Lambda}^k(\overline{f_\xi})\)
by \begin{equation} \label{eq:ring-xi}
\mathring{E}_{\sigma,\xi}: \omega_\sigma
\mapsto P_{\xi,\sigma}^* \omega_\sigma \wedge (\lambda_{\sigma}^{(\xi)}
(\mathrm{d}\lambda^{(\xi)})_{\xi \setminus \sigma}),
\end{equation} and we have a similar decomposition result to
\eqref{eq:bubble-decomp}, \[
\mathring{\Lambda}^k(\overline{f_\xi})
= \bigoplus_{\sigma \in Z^k(f_\xi)} \mathring{E}_{\sigma,\xi} [
\Lambda^{k - (\dim(f_\xi) - \dim(f_\sigma))}(\overline f_\sigma)
].
\] We use the bubble decomposition to define the extension operator
\(\dot{E}_{\tau,\xi}: \mathring{\Lambda}^{k}(\overline{f_\sigma}) \to \Lambda^k (\overline {f_\xi})\)
when \([\![\tau ]\!]\subseteq [\![\xi ]\!]\) by its action on the bubble
components: \begin{equation} \label{eq:dotty-xi}
\dot{E}_{\tau,\xi}: \mathring{E}_{\sigma,\tau} \omega_\sigma \mapsto
P_{\xi,\sigma}^* \omega_\sigma \wedge (\lambda_{\sigma}^{(\xi)}
\wedge
(\mathrm{d}\lambda^{(\xi)})_{\tau \setminus \sigma}).
\end{equation} Note the difference between \eqref{eq:ring-xi} and
\eqref{eq:dotty-xi} is only in the barycentric form wedged with the
pullback: for
\(\dot{E}_{\tau,\xi} \mathring{E}_{\sigma,\tau} \omega_\sigma\) it is
\((\mathrm{d}\lambda^{\xi})_{\tau \setminus \sigma}\), so that the
product
\(\lambda_\sigma^{(\xi)} (\mathrm{d}\lambda^{(\xi)})_{\tau \setminus \sigma}\)
is \emph{not} trace-free on \(f_\xi\), but has nonzero trace on any
boundary simplex of \(f_\xi\) that contains \(f_\tau\).

\begin{theorem}
Let a family of extension operators be defined by
$$\{
\dot{E}_{\tau,\xi}: \mathring{\Lambda}^k(\overline{f_\tau})
\to \Lambda^k(\overline{f_\xi}), [\![\tau ]\!] \subseteq [\![\xi ]\!]
\},$$
where $\dot{E}_{\tau,\xi}$ is defined \eqref{eq:dotty-xi}.  The family
is a consistent family, that is for simplices $f_\rho$, $f_\tau$ and
$f_\xi$ such that $[\![\rho ]\!],[\![\tau ]\!] \subseteq [\![\xi ]\!]$,
we have
\begin{equation}\label{eq:consistent-family}
\mathrm{Tr}_{\xi,\rho} \dot{E}_{\tau,\xi} \omega = 
\dot{E}_{\rho \cap \tau, \rho} \mathrm{Tr}_{\tau,\rho\cap\tau} \omega,
\quad \omega \in \mathring{\Lambda}^k(\overline{f_\sigma}).
\end{equation}
\end{theorem}

\begin{proof}
Let $\xi$ be a $d$-simplex and let $\tau$ be a boundary simplex, $[\![\tau ]\!] \subseteq [\![\xi ]\!]$.  Let $f_\sigma$ be any boundary simplex of $f_\tau$ such that
$\dim(f_\sigma) \geq \dim(f_\tau) - k$, and let $\omega_\sigma
\in \Lambda^{k - (\dim(f_\tau) - \dim(f_\sigma))} (\overline{f_\sigma})$
be given.  $\mathring{E}_{\sigma,\tau} \omega_\sigma$ is a $k$-form
in $\mathring{\Lambda}^k (\overline{f_\sigma})$.

First let a simplex $f_\rho$ be given such that $[\![\tau ]\!] \not\subseteq
[\![\rho ]\!]$.  Because $\mathring{E}_{\sigma,\tau} \omega_\sigma$
is trace-free, $\mathrm{Tr}_{\tau,\rho\cap\tau}\mathring{E}_{\sigma,\tau} \omega_\sigma = 0$, and so
$$\dot{E}_{\rho \cap \tau, \rho} \mathrm{Tr}_{\tau,\rho\cap\tau}
\mathring{E}_{\sigma,\tau} \omega_\sigma = 0.$$
By the definition in \eqref{eq:dotty-xi}, we see that
$\mathrm{Tr}_{\xi,\rho} \dot{E}_{\tau,\xi} \mathring{E}_{\sigma,\tau}
\omega_\sigma = 0$ as well, because $\mathrm{Tr}_{\xi,\rho}\lambda_{\sigma}^{(\xi)} (\mathrm{d}\lambda^{(\xi)})_{\tau \setminus \sigma} = 0.$

Now let a simplex $f_\rho$ be given such that $[\![\tau ]\!] \subseteq
[\![\rho ]\!]$, which implies $\mathrm{Tr}_{\tau,\rho\cap\tau}
\mathring{E}_{\sigma,\tau} \omega_\sigma = \omega_\sigma.$
We compute the left hand side of \eqref{eq:consistent-family}:
$$
\begin{aligned}
\mathrm{Tr}_{\xi,\rho} \dot{E}_{\tau,\xi}( \mathring{E}_{\sigma,\tau} \omega_\sigma) 
&=
\mathrm{Tr}_{\xi,\rho} (P_{\xi,\sigma}^* \omega_\sigma \wedge
(\lambda^{(\xi)}_{\sigma} (\mathrm{d}\lambda^{(\xi)})_{\tau \setminus \sigma})
)
\\
&=
(\mathrm{Tr}_{\xi,\rho} P_{\xi,\sigma}^* \omega_\sigma) \wedge
(\mathrm{Tr}_{\xi,\rho} \lambda^{(\xi)}_{\sigma} (\mathrm{d}\lambda^{(\xi)})_{\tau \setminus \sigma})
\\
&=
(P_{\rho,\sigma}^* \omega_\sigma) \wedge
(\lambda^{(\rho)}_{\sigma} (\mathrm{d}\lambda^{(\rho)})_{\tau \setminus \sigma})
\\
&=
\dot{E}_{\tau,\rho} \mathring{E}_{\sigma,\tau} \omega_\sigma
=
\dot{E}_{\rho \cap \tau, \rho} \mathrm{Tr}_{\tau, \rho \cap \tau}
(\mathring{E}_{\sigma,\tau} \omega_\sigma),
\end{aligned}
$$
where we have used \eqref{eq:tr-inv-p} between the second and third right hand sides.
Because every $\omega \in \mathring{\Lambda}^k(\overline{f_\tau})$ has a bubble decomposition, this completes the proof.
\end{proof}

\hypertarget{geometric-decompositions}{%
\subsection{Geometric decompositions}\label{geometric-decompositions}}

With \(\dot{E}_{\sigma,\tau}\) defining a consistent family of extension
operators, we can now show that it defines a geometric decomposition of
all of \(\Lambda^k(\overline T)\).

\begin{theorem}
With $\dot{E}_{\sigma,T}$ defined as in \eqref{eq:dotty-xi} taking $f_\xi = T$,
$$
\Lambda^k(\overline T) =
\bigoplus_{\substack{\sigma \\ \dim(f_\sigma) \geq k}}
\dot{E}_{\sigma,T}[ \mathring{\Lambda}^k(\overline{f_\sigma}) ].
$$
\end{theorem}

\begin{proof}
The proof proceeds along the same lines as the proof of the bubble decompostion in \cref{thm:ringe-biject}. Implicit in the fact that
$\dot{E}_{\sigma,\tau}$ defines a consistent family of extension operators is that
$$
\mathrm{Tr}_{\sigma} \dot{E}_{\sigma,T} \omega = \omega, \quad \omega \in \mathring{\Lambda}^k(\overline{f_\sigma}).
$$
This means that if  $\omega \in \mathring{\Lambda}^k(\overline{T})$
then $\omega = \dot{E}_{T,T} \omega$.
From this base case one can inductively assume
that every $\omega$ has a geometric decomposition in the $\dot{E}_{\sigma,\tau}$ family if
it is trace-free on boundary simplices with dimension less than or equal to $d$.  Then taking $\omega$ that is trace-free on boundary simplicies with dimension less than $d$ and define
$$
\tilde{\omega} = \omega - \sum_{\sigma \in \Sigma([0\range{d}],[0\range{n}])} \dot{E}_{\sigma,T}
\mathrm{Tr}_{\sigma} \omega,
$$
which by the inductive assumption has a geometric decomposition in the $\dot{E}_{\sigma,\tau}$ family, and so $\omega$ has one as well.  By the inductive assumption, this completes the proof.
\end{proof}

We now show that the \(\dot{E}_{\sigma,\tau}\) extension operator
defines a geometric decomposition of full polynomial \(k\)-forms.

\begin{theorem}\label{thm:dotty-full}
With $\dot{E}_{\sigma,T}$ defined as in \eqref{eq:dotty-xi} taking $f_\xi = T$,
$$
\mathcal{P}_r \Lambda^k(T) =
\bigoplus_{\substack{\sigma \\ \dim(f_\sigma) \geq k}}
\dot{E}_{\sigma,T}[ \mathring{\mathcal{P}}_r \Lambda^k(f_\sigma) ].
$$
\end{theorem}

\begin{proof}
Let $f_\tau$ such that $\dim(f_\tau) \geq k$ be given.
By \eqref{eq:bubble-full}, $\mathcal{P}_r \Lambda^k(f_\tau)$ is spanned by bubble functions of the form
$\mathring{E}_{\sigma,\tau} \omega_{\sigma}$ for $\omega_\sigma
\in \mathcal{P}_{r - \dim(f_\sigma) - 1} \Lambda^{k - (\dim(f_\tau) - \dim(f_\sigma))}(f_\sigma).$
By inspection of \eqref{eq:dotty-xi} we see that this means
$\dot{E}_{\sigma,T} \mathring{E}_{\sigma,\tau} \omega_{\sigma}\in
\mathcal{P}_r \Lambda^k(T).$
The injectivity of $\dot{E}_{\sigma,T}$, the fact that it defines a consistent family of extension operators, and a counting argument
are sufficient to complete the proof.
\end{proof}

Our final theorem shows that \(\dot{E}_{\sigma,\tau}\) defines a
geometric decomposition of trimmed polynomial \(k\)-forms as well.

\begin{theorem}
With $\dot{E}_{\sigma,T}$ defined as in \eqref{eq:dotty-xi} taking $f_\xi = T$,
$$
\mathcal{P}_r^- \Lambda^k(T) =
\bigoplus_{\substack{\sigma \\ \dim(f_\sigma) \geq k}}
\dot{E}_{\sigma,T}[ \mathring{\mathcal{P}}_r^- \Lambda^k(f_\sigma) ].
$$
\end{theorem}

\begin{proof}
It is sufficient to show that
$\dot{E}_{\tau,T} \omega \in \mathcal{P}_r^- \Lambda^k(T)$ for
each $f_\tau$ such that $\dim(f_\tau) \geq k$ and each
$\omega \in \mathring{\mathcal{P}}_r^- \Lambda^k(f_\tau)$.
By the proof of \cref{thm:dotty-full},
it is already known that
$\dot{E}_{\tau,T} \omega \in \mathcal{P}_r \Lambda^k(T)$.
We shall apply the Koszul operator centered at the centroid
of $f_\tau$ to $\dot{E}_{\tau,T} \omega$: if we can show that
$\kappa^{(\tau)} \dot{E}_{\tau,T} \omega \in
\mathcal{P}_r \Lambda^{k-1}(T)$, that is sufficient to prove that
$\dot{E}_{\tau,T} \omega \in \mathcal{P}_r^- \Lambda^k(T)$.

Let the bubble decomposition of $\omega$ be
$$
\omega = \sum_{\sigma \in Z^k(f_\tau)}
\mathring{E}_{\sigma,\tau} \omega_\sigma.
$$
We apply $\kappa^{(\tau)}$ to the extension of a single bubble component, but shift
it to be centered at $\kappa^{(\sigma)}$, which introduces a term
involving the interior product with some constant vector $b_{\sigma,\tau}$, which we
denote $\delta_{\sigma,1}$:
$$
\kappa^{(\tau)} \dot{E}_{\tau,T} \mathring{E}_{\sigma,\tau} \omega_\sigma
=
\kappa^{(\sigma)} \dot{E}_{\tau,T} \mathring{E}_{\sigma,\tau} \omega_\sigma
+
\underbrace{(\dot{E}_{\tau,T} \mathring{E}_{\sigma,\tau} \omega_\sigma) 
\lrcorner b_{\sigma,\tau}}_{\delta_{\sigma,1}}.
$$
Because the vector $b_{\sigma,\tau}$ is constant, the $(k-1)$-form $\delta_{\sigma,1}$ is in $\mathcal{P}_r \Lambda^{k-1}(T)$.

We now address the term $\kappa^{(\sigma)} \dot{E}_{\tau,T} \mathring{E}_{\sigma,\tau} \omega_\sigma$.  We wish to show that in this case $\kappa^{(\sigma)}$ and $\dot{E}_{\tau,T}$ commute.  Using the product rule for the Koszul operator, we have
\begin{align}
\kappa^{(\sigma)} \dot{E}_{\tau,T} \mathring{E}_{\sigma,\tau} \omega_\sigma
&=
\kappa^{(\sigma)}
[P_{T,\sigma}^* \omega_\sigma \wedge (\lambda_\sigma (\mathrm{d}\lambda)_{\tau \setminus \sigma})]
\\
&=
(\kappa^{(\sigma)}P_{T,\sigma}^* \omega_\sigma) \wedge (\lambda_\sigma (\mathrm{d}\lambda)_{\tau \setminus \sigma})
+(-1)^{\tilde{k}}
P_{T,\sigma}^* \omega_\sigma \wedge (\lambda_\sigma \kappa^{(\sigma)}(\mathrm{d}\lambda)_{\tau \setminus \sigma}),\label{eq:koszul-two-terms}
\end{align}
where $\tilde{k} = k - (\dim(f_\tau) - \dim(f_\sigma))$.
The centroid of $f_\sigma$ is fixed by $P_{T,\sigma}$, so $\kappa^{(\sigma)}$ commutes with $P_{T,\sigma}^*$ and the first term becomes
$$
\begin{aligned}
(\kappa^{(\sigma)}P_{T,\sigma}^* \omega_\sigma) \wedge (\lambda_\sigma (\mathrm{d}\lambda)_{\tau \setminus \sigma})
&=
P_{T,\sigma}^*(\kappa^{(\sigma)} \omega_\sigma) \wedge (\lambda_\sigma (\mathrm{d}\lambda)_{\tau \setminus \sigma})
\\
&=
\dot{E}_{\tau,T} \mathring{E}_{\sigma,\tau} (\kappa^{(\sigma)} \omega_\sigma).
\end{aligned}
$$
Because $\lambda_j$ vanishes at the centroid of $f_\sigma$ if
$j \in [\![\tau \setminus \sigma ]\!]$,
we have $\kappa^{(\sigma)} \mathrm{d}\lambda_j = \lambda_j$ for each
such $j$.  Let $m+1$ be the number of indices in $[\![\tau \setminus \sigma ]\!]$ and let $\rho \in \Sigma([0\range{m}],[0\range{n}])$ be an increasing map such that
$[\![\rho ]\!] = [\![\tau \setminus \sigma ]\!]$.
By the product rule
$$
\kappa^{(\sigma)} (\mathrm{d}\lambda)_{\tau \setminus \sigma}
= 
\sum_{j = 0}^m (-1)^m \lambda_{\rho(j)} (\mathrm{d}\lambda)_{\rho \setminus \rho(j)}.
$$
Putting this into the second term in \eqref{eq:koszul-two-terms},
we get
$$
\begin{aligned}
(-1)^{\tilde{k}}
P_{T,\sigma}^* \omega_\sigma \wedge (\lambda_\sigma \kappa^{(\sigma)}(\mathrm{d}\lambda)_{\tau \setminus \sigma})
&=
\sum_{j=0}^m (-1)^{j + \tilde{k}}
P_{T,\sigma}^* \omega_\sigma \wedge (\lambda_{\sigma\cup\rho(j)} (\mathrm{d}\lambda)_{\rho \setminus \rho(j)})
\\
&=
\sum_{j=0}^m 
P_{T,\sigma\cup\rho(j)}^* P_{\sigma\cup\rho(j),\sigma}^* ((-1)^{j + \tilde{k}}\omega_\sigma) \wedge (\lambda_{\sigma\cup\rho(j)} (\mathrm{d}\lambda)_{\rho \setminus \rho(j)})
\\
&=
\dot{E}_{\tau,T}\big[
\sum_{j=0}^m 
\mathring{E}_{\sigma \cup\rho(j),\tau}(P_{\sigma\cup\rho(j),\sigma}^*((-1)^{j + \tilde{k}}\omega_\sigma)) 
\big].
\end{aligned}
$$
Combining both terms, we have
$$
\kappa^{(\sigma)} \dot{E}_{\tau,T} \mathring{E}_{\sigma,\tau} \omega_\sigma
=
\dot{E}_{\tau,T}\big[
\mathring{E}_{\sigma,\tau} (\kappa^{(\sigma)} \omega_\sigma) +
\sum_{j=0}^m 
\mathring{E}_{\sigma \cup\rho(j),\tau}(P_{\sigma\cup\rho(j),\sigma}^*(-1)^{j + \tilde{k}}\omega_\sigma) 
\big].
$$
We now want to show that the term in brackets is $\kappa^{(\sigma)}
\mathring{E}_{\sigma,\tau} \omega_\sigma.$
We essentially follow the steps above in reverse.  For the first term,
$$
\begin{aligned}
\mathring{E}_{\sigma,\tau} (\kappa^{(\sigma)} \omega_\sigma)
&=
P_{\tau,\sigma}^* (\kappa^{(\sigma)} \omega_\sigma)
\wedge (\lambda^{(\tau)}_{\sigma} (\mathrm{d}\lambda^{(\tau)})_{\tau\setminus\sigma})
\\
&=
(\kappa^{(\sigma)} P_{\tau,\sigma}^* \omega_\sigma)
\wedge (\lambda^{(\tau)}_{\sigma} (\mathrm{d}\lambda^{(\tau)})_{\tau\setminus\sigma}).
\end{aligned}
$$
For the second term,
$$
\begin{aligned}
\sum_{j=0}^m 
\mathring{E}_{\sigma \cup\rho(j),\tau}(P_{\sigma\cup\rho(j),\sigma}^*((-1)^{j + \tilde{k}}\omega_\sigma)) 
&=
\sum_{j=0}^m 
P_{\tau,\sigma\cup\rho(j)}^* P_{\sigma\cup\rho(j),\sigma}^* ((-1)^{j + \tilde{k}}\omega_\sigma) \wedge (\lambda_{\sigma\cup\rho(j)}^{(\tau)} (\mathrm{d}\lambda^{(\tau)})_{\rho \setminus \rho(j)})
\\
&=
\sum_{j=0}^m (-1)^{j + \tilde{k}}
P_{\tau,\sigma}^* \omega_\sigma \wedge (\lambda_{\sigma\cup\rho(j)}^{(\tau)} (\mathrm{d}\lambda^{(\tau)})_{\rho \setminus \rho(j)})
\\
&=
(-1)^{\tilde{k}} 
P_{\tau,\sigma}^* \omega_\sigma \wedge
(\lambda_{\sigma}^{(\tau)} \kappa^{(\sigma)} (\mathrm{d}\lambda^{(\tau)})_{\tau \setminus \sigma}).
\end{aligned}
$$
Combining these back together we get the desired result,
$$
(\kappa^{(\sigma)} P_{\tau,\sigma}^* \omega_\sigma)
\wedge (\lambda^{(\tau)}_{\sigma} (\mathrm{d}\lambda^{(\tau)})_{\tau\setminus\sigma})
+
(-1)^{\tilde{k}} 
P_{\tau,\sigma}^* \omega_\sigma \wedge
(\lambda_{\sigma}^{(\tau)} \kappa^{(\sigma)} (\mathrm{d}\lambda^{(\tau)})_{\tau \setminus \sigma})
=
\kappa^{(\sigma)} \mathring{E}_{\sigma,\tau} \omega_\sigma,
$$
that is,
$$
\kappa^{(\sigma)} \dot{E}_{\tau,T} \mathring{E}_{\sigma,\tau} \omega_\sigma
=
\dot{E}_{\tau,T} \kappa^{(\sigma)} \mathring{E}_{\sigma,\tau} \omega_\sigma.
$$
We now shift the center of the Koszul operator back to $\kappa^{(\tau)}$, 
which again introduces an interior product with the constant vector
which we denote $\delta_{\sigma,2}$:
$$
\dot{E}_{\tau,T} \kappa^{(\sigma)} \mathring{E}_{\sigma,\tau} \omega_\sigma
=
\dot{E}_{\tau,T}(\kappa^{(\tau)} \mathring{E}_{\sigma,\tau} \omega_\sigma
-
\underbrace{(\mathring{E}_{\sigma,\tau} \omega_\sigma)\lrcorner b_{\sigma,\tau}}_{\delta_{\sigma,2}}).
$$
As with the $\delta_{\sigma,1}$, we know that $\delta_{\sigma,2} \in \mathcal{P}_r \Lambda^{k-1}(f_\sigma)$.

Why have we gone through the steps of shifting the Koszul operator
between the centroids of $f_{\sigma}$ and $f_{\tau}$?
Because although $\kappa^{(\sigma)} \mathring{E}_{\sigma,\tau} \omega_\sigma$ is trace-free (as demonstrated in the steps that it took to show that $\kappa^{(\sigma)}$ and $\dot{E}_{\tau,T}$ commuted), $\kappa^{(\tau)} \mathring{E}_{\sigma,\tau} \omega_\sigma$ is not necessarily trace-free.
We cannot apply $\dot{E}_{\tau,T}$ to 
$\kappa^{(\tau)} \mathring{E}_{\sigma,\tau} \omega_\sigma$ or $\delta_{\sigma,2}$ separately, because the domain of $\dot{E}_{\tau,T}$
is trace-free $k$-forms.

Nevertheless,
we can now combine into $\kappa^{(\tau)} \dot{E}_{\tau,T} \omega$ the
contributions from the different terms in the bubble decomposition:
$$
\begin{aligned}
\kappa^{(\tau)} \dot{E}_{\tau,T} \omega
&=
\sum_{\sigma \in Z^k(f_\tau)}
\dot{E}_{\tau,T}(\kappa^{(\tau)}\mathring{E}_{\sigma,T} \omega_\sigma + \delta_{\sigma,2})
+
\delta_{\sigma,1}
\\
&=
\dot{E}_{\tau,T}\left[
\kappa^{(\tau)}\omega + 
\sum_{\sigma \in Z^k(f_\tau)}
\delta_{\sigma,2}
\right]
+
\sum_{\sigma \in Z^k(f_\tau)}
\delta_{\sigma,1}.
\end{aligned}
$$
By the assumption that $\omega \in \mathring{\mathcal{P}}_r^- \Lambda^{k}(f_\sigma)$,
we know that $\kappa^{(\tau)}\omega \in \mathcal{P}_r \Lambda^{k-1}
(f_\sigma)$.
Therefore the term in brackets is in 
$\mathring{\mathcal{P}}_r \Lambda^{k-1}(f_\sigma)$.
By \cref{thm:dotty-full}, that implies that the first term,
and thus the whole expression, is in 
$\mathcal{P}_r\Lambda^{k-1}(T)$.
\end{proof}

The main results of this work combine to unify the geometric
decompositions of full polynomials, trimmed polynomials, and all
\(k\)-forms through the action of just two operators acting on the
unrestricted polynomial spaces.

\begin{corollary}
Generalizing the operator $\mathring{\star}_T$ to $\mathring{\star}_{\sigma}:\Lambda^k(\overline{f_\sigma}) \stackrel{\sim\,}{\to}
\mathring{\Lambda}^{\dim(f_\sigma)-k}(\overline{f_\sigma})$,
we have
\begin{align}
\Lambda^k(\overline{T}) &=
\bigoplus_{\substack{\sigma \\ \dim(f_\sigma) \geq k}}
\dot{E}_{\sigma,T}\ \mathring{\star}_{\sigma} [ \Lambda^{\dim(f_\sigma) - k} (\overline{f_\sigma})];
\\
\mathcal{P}_r \Lambda^k(T)
&=
\bigoplus_{\substack{\sigma \\ \dim(f_\sigma) \geq k}}
\dot{E}_{\sigma,T}\ \mathring{\star}_{\sigma} [ \mathcal{P}_{r - \dim(f_\sigma) - k}^- \Lambda^{\dim(f_\sigma) - k} (f_\sigma)];
\\
\mathcal{P}_r^- \Lambda^k(T)
&= 
\bigoplus_{\substack{\sigma \\ \dim(f_\sigma) \geq k}}
\dot{E}_{\sigma,T}\ \mathring{\star}_{\sigma} [ \mathcal{P}_{r - \dim(f_\sigma) - k - 1} \Lambda^{\dim(f_\sigma) - k} (f_\sigma)].
\end{align}
\end{corollary}

\printbibliography[title=References]

\end{document}